\documentclass[12pt]{amsart}
\usepackage{color}
\usepackage{amssymb}
\usepackage{comment}
\usepackage{pdfpages}
\usepackage{graphicx}% Include figure files
\usepackage{dcolumn}% Align table columns on decimal point
\RequirePackage[numbers]{natbib}
\RequirePackage[colorlinks=true, pdfstartview=FitV, linkcolor=blue,
  citecolor=blue, urlcolor=blue]{hyperref}
\RequirePackage{hypernat}
\usepackage{paralist}

\usepackage{graphicx}
\graphicspath{{./figures/}}
\usepackage{amsfonts}
\usepackage{amsmath}
\usepackage{amsthm}
\usepackage{amssymb}
\usepackage{amsbsy}
\usepackage{epsfig}
\usepackage{fullpage}
\usepackage{natbib, mathrsfs} 
\usepackage{verbatim}
\usepackage[latin1]{inputenc}
\usepackage{mhequ}
\usepackage{algorithm}
\usepackage{algorithmic}

\numberwithin{equation}{section}

\def \be{\begin{equs}}
\def \ee{\end{equs}}
\def \P{\mathbb{P}}
\def \E{\mathbb{E}}

\newcommand \TV{\mathrm{TV}}

\def \tmix{\tau_{\mathrm{mix}}}
\def \ttmix{\widetilde{\tau}_{\mathrm{mix}}}

\def \TV{\mathrm{TV}}
\def \var{\mathrm{Var}}
\def \KMH{K^{(\mathrm{MH})}}
\def \PMH{\pi^{(\mathrm{MH})}}

\newcommand \KS[1]{K^{(\mathrm{Simp-MH})}_{#1}}
\newcommand \KSW[2]{K^{(\mathrm{Simp-Wide})}_{#1, #2}}

\newcommand \KMHW[1]{K^{(\mathrm{MH-Wide})}_{#1}}
\newcommand \PMHW[1]{\pi^{(\mathrm{MH-Wide})}_{#1}}

\newcommand \KMHI[1]{K^{(\mathrm{Lim})}_{#1}}
\newcommand \PMHI[1]{\pi^{(\mathrm{Lim})}_{#1}}

\newtheorem{theorem}{Theorem}[section]
\newtheorem{lemma}[theorem]{Lemma}

\newtheorem{assumption}[theorem]{Assumption}

\theoremstyle{plain}
\newtheorem{thm}{Theorem}
\newtheorem*{thm-non}{Theorem}
\newtheorem{cor}[thm]{Corollary}

\theoremstyle{definition}
\newtheorem{example}[theorem]{Example}

\newtheorem{remark}[theorem]{Remark}

\begin{document}

%\begin{frontmatter}
% "Title of the paper"
\title[Ergodicity of Approximate Algorithms]
{Ergodicity of Approximate MCMC Chains with applications to Large Data Sets}

% indicate corresponding author with \corref{}
% \author{\fnms{John} \snm{Smith}\corref{}\ead[label=e1]{smith@foo.com}\thanksref{t1}}
 %\thankstext{t1}{Thanks to somebody} 
 %\address{line 1\\ line 2\\ printead{e1}}
% \affiliation{Some University}

\author{Natesh S. Pillai$^{\ddag}$}
\thanks{$^{\ddag}$pillai@fas.harvard.edu, 
   Department of Statistics
    Harvard University, 1 Oxford Street, Cambridge
    MA 02138, USA}

\author{Aaron Smith$^{\sharp}$}
\thanks{$^{\sharp}$smith.aaron.matthew@gmail.com, 
   Department of Mathematics and Statistics
University of Ottawa, 585 King Edward Drive, Ottawa
ON K1N 7N5, Canada}
 % \thanks{NSP is partially supported by NSF-DMS 1107070}

\maketitle
\begin{abstract}
In many modern applications, difficulty in evaluating the posterior density makes performing even a single MCMC step slow. This difficulty can be caused by intractable likelihood functions, but also appears for routine problems with large data sets. Many researchers have responded by running approximate versions of MCMC algorithms. In this note, we develop quantitative bounds for showing the ergodicity of these approximate samplers. We then use these bounds to study the bias-variance trade-off of approximate MCMC algorithms. We apply our results to simple versions of recently proposed algorithms, including a variant of the ``austerity" framework of Korratikara et al.
\end{abstract}

% AMS subject classifications (used in AMS journals)
   %\subjclass{Primary 60J22; Secondary 60J20, 60H15, 65C40}

   % AMS keywords (used in AMS journals)
 % \keywords{Approximate MCMC, Austerity framework, Stochastic Gradient Langevin, ABC-MCMC}

%\begin{keyword}[class=AMS]
%\kwd[Primary ]{1232}
%\kwd{ser}
%\kwd[; secondary ]{123}
%\end{keyword}

%\begin{keyword}
%\kwd{}
%\kwd{}
%\end{keyword}

%\end{frontmatter}

% AOS,AOAS: If there are supplements please fill:
%\begin{supplement}[id=suppA]
%  \sname{Supplement A}
%  \stitle{Title}
%  \slink[url]{http://lib.stat.cmu.edu/aoas/???/???}
%  \sdescription{Some text}
%\end{supplement}

\section{Introduction}
Markov chain Monte Carlo (MCMC) sampling is an indispensable tool for Bayesian computation. Most Metropolis-Hastings samplers require full evaluation of the posterior at two points at every step; some MCMC samplers require even more information, such as gradients of the likelihood function. In many modern applications, the computational cost of this full evaluation of the likelihood function can be prohibitively large. For instance, a full likelihood evaluation might involve processing a massive amount data, or computing the solution of a partial differential equation representing an underlying physical phenomenon. The result is that the inferential performance of naive implementations of MCMC algorithms can actually deteriorate as the amount of data grows, unless available computational resources grow even more quickly. The same problem has been noted for other techniques that are popular in computational statistics (see \textit{e.g.,} \cite{Jord13, SVA15} for broader discussions of this problem in a similar spirit, outside of the context of MCMC). As mentioned in those papers, it might be possible to avoid actual loss of performance if we are aware that it is a possibility. However, it is much harder to ensure that data is being used \textit{efficiently} under severe computational constraints. \par 

Simulating a Markov chain that merely approximates the desired MCMC dynamics is an increasingly common approach to circumvent this difficulty. These approximations often rely on estimating, rather than evaluating, the posterior distribution of interest - for example, doing so based on a subsample of the available data \cite{Beau03, OBBEM00, WeTe11, KCW13, BDC14, quiroz2014speeding, AFEB14, maire2015light}. While many approximate MCMC methods seem to be very successful in practice, they do not have the same convergence guarantees as standard MCMC samplers. In this paper, we present some general convergence results for such `approximate' MCMC algorithms and discuss their applications to some recently proposed algorithms. Our results give quantitative bounds on the convergence in distribution of the Markov chain as well as convergence of finite samples drawn from the Markov chain. These bounds allow us to provide advice on how to choose parameters for various approximation schemes. As a consequence, they also give modest conditions under which certain approximation schemes are more efficient than their underlying MCMC algorithms. This result is unsurprising, but it is quite hard to prove for many existing approximation schemes for reasons discussed at the start of Section \ref{SecAusterity}. While our applications focus on the problems posed by large data sets, our theoretical bounds are also relevant to MCMC samplers targetting intractable likelihoods; see \cite{MCPS13} for applications of related ideas in that context. \par 

\subsection{Our Contribution}
Our paper has three main contributions. The first is providing a variety of robust ergodicity results for perturbed Markov chains, comparing the  mixing properties and Monte Carlo errors of a perturbed chain to its base chain. The second is an application of these bounds to certain `interpolating' chains in order to obtain a bias-variance tradeoff inequality. This inequality provides some advice on how to choose the best approximate sampler for a given computational budget. In particular, it gives sufficient conditions under which particular approximate MCMC samplers give smaller Monte Carlo errors than the `correct' Metropolis-Hastings dynamics.  The third is a discussion of the limits of our perturbation-based analysis for studying the tradeoff between computational and statistical efficiency for Markov chains, with a focus on some `obvious' facts that cannot be proved in this framework.   \par

In the first part of the paper, we are especially interested in how approximations perform outside of the simplest setting of uniformly good approximations of uniformly ergodic Markov chains (see \textit{e.g.}, \cite{Mitr05} for very good bounds in that setting;  heuristic arguments presented in \cite{KCW13} and \cite{WeTe11} implicitly make such assumptions). It is well known that some approximation schemes that have been proposed in the past can fail badly even in innocuous settings when the approximations are not uniformly good (see Example 17 in \cite{MCPS13} for a sampler that fails to converge to a two-valued density on $[0,1]$ and Theorem 1 of \cite{LeLa12} for one that has difficulty sampling from a Gaussian). Similarly, even uniformly good approximations of chains without uniform ergodicity can fail to inherit good convergence properties. To see this, fix $\epsilon > 0$ and consider simple random walk on $\mathbb{N}$ as a uniformly good approximation of simple random walk on $\mathbb{N}$ with drift $| \frac{\epsilon}{\log(n)}|$ towards 0 at point $n \geq 2$. The latter chain has reasonable convergence properties; the former does not even have a stationary measure.   
\subsection{Guide to the Paper}
We begin by setting up some notation in Section \ref{SecNotation}, and prove our general perturbation bounds in Section \ref{SecConvBounds}. We apply our bounds to the \textit{austerity} framework of \cite{KCW13} in Section \ref{SecAusterity}. This application begins with an introduction to the austerity framework and some definitions related to the idea of `computational complexity' of MCMC. We continue by providing a general bias-variance tradeoff result in this framework. This bound provides sufficient conditions under which an approximate MCMC kernel is better than the true Metropolis-Hastings kernel. In Sections \ref{SecCompChoice} and \ref{SecJust} we note that the perturbation results in Section \ref{SecConvBounds} give very poor bounds if applied directly to the true Metropolis-Hastings dynamics, and introduce interpolating chains and new perturbation bounds that allow us to avoid this problem. In Section \ref{SecBigMisc}, we discuss the sharpness of our results and point out some interesting phenomena that seem difficult to study via perturbation analysis.

\subsection{Related Literature}
The earlier papers \cite{KCW13, Mitr05, FHL13} also give abstract convergence results for perturbations of Markov chains, though the first two primarily restrict their attention to uniformly ergodic chains and the latter does not give the required quantitative bounds for the types of comparisons we do in this paper. 

Since writing our first version of this paper, a number of other preprints and articles studying the perturbation theory of Markov chains and applications to `big data' have been released. This includes \cite{BDC14, AFEB14, quiroz2014speeding, zhu2014big,friel2015exploiting,rudolf2015perturbation,maire2015light,bardenet2015markov,chen2014sublinear,chen2015subsampling,green2015bayesian,johndrow2015approx}, the latter 9 of which refer to our earlier draft. We briefly discuss some relationships between these papers and our work.

In \cite{BDC14,AFEB14}, the authors discuss or introduce various approximate MCMC algorithms and prove that they are in fact small perturbations of the `true' MCMC dynamics under various conditions. This justifies the use of perturbation analysis to bound the bias of these approximate MCMC samplers. However, when applied to subsampling MCMC algorithms, the bounds in both papers typically apply only when the average subsample size $n$ is on the same order as the total amount of data $N$. \cite{quiroz2014speeding} introduces another subsampling MCMC algorithm and gives the first bounds on the decay rate of the bias of their algorithm that depends only on the subsample size $n$, not the total amount of data $N$. However, these bounds do not give an explicit decay rate in terms of $n$. The excellent survey article on approximate MCMC \cite{bardenet2015markov} also provides arguments on the convergence of subsampling MCMC algorithms when $n \ll N$, among many other results.\par 
The recent article \cite{johndrow2015approx} introduces several new approximate MCMC algorithms, and also has some focus on testing when their approximate MCMC algorithms are `better' than the default MCMC methods. This is an important aspect that we do not address in our paper. Understanding this issue for a wide class of algorithms in the key next step. The authors in \cite{johndrow2015approx} provide a criterion that is very similar to the criterion introduced in Section \ref{SubsecMeasCompStatEff} of this paper; but their criteria makes sense even in the regime $n \ll N$.\par \cite{rudolf2015perturbation} gives bounds on convergence of perturbed Markov chains under very general conditions, the main subject of our Section \ref{SecConvBounds}. Some of their bounds, like some of ours, are based on the `curvature' of the original Markov chain.  However, they consider many more technical conditions than we do.The other papers cited are also about approximate MCMC, but have core messages that do not greatly overlap with ours. We emphasize that some of these subsampling algorithms (\textit{e.g.,} \cite{maire2015light, chen2014sublinear}) do not obviously require $n \approx N$ in order to be useful.

Of these papers, we feel that results in \cite{quiroz2014speeding} are closest in spirit to our conclusions.  Our main purpose in writing this note was to show that subsampling MCMC schemes can give much smaller Monte Carlo errors than `correct' Metropolis-Hastings schemes, even when the amount of data $N$ is very large to the total computational resources available. Among other contributions, \cite{quiroz2014speeding} showed that this was plausible by bounding the bias of subsampling MCMC in a way that depended only on the subsample size $n$, not the total amount of data available $N$. They also discussed notions of computational complexity that are similar to ours. Our analysis differs from theirs in a number of ways. Most obviously, the bias bounds in \cite{quiroz2014speeding} were based on the pseudomarginal algorithm; our bounds apply to approximate chains that are not pseudomarginal. Technically, our bounds are obtained via perturbation theory, and we introduce new `interpolating' chains that allow the powerful perturbation techniques to be applied to the small-$n$, large-$N$ regime. Finally, we bound both the bias and the mixing properties of approximate chains, while \cite{quiroz2014speeding}  focuses on the bias.\par
Finally, we compare and contrast our results to that of \cite{AFEB14}.  Most of the perturbation bounds in \cite{AFEB14} are closely related to 
those in \cite{Mitr05}. Although our general bounds are similar in spirit to \cite{AFEB14}, our arguments are quite different and there are situations where our bounds apply and theirs do not (and vice versa). See Remark \ref{RemarkCompFriel} for a longer discussion on the relationship between our results and that of \cite{Mitr05}.

\section{Notation} \label{SecNotation}
For a random variable $X$ and measure $\mu$, we write $X \sim \mu$ or $\mathcal{L}(X) = \mu$ to mean that $X$ is distributed according to $\mu$. $\mathrm{Unif}(A)$ denotes the uniform or Haar distribution on the set $A$ as appropriate. We will write $f = O(g)$ to mean that there exists a constant $C > 0$ so that $f(x) \leq C g(x)$. We also write $f = o(g)$ if $\lim_{x \rightarrow \infty} \frac{f(x)}{g(x)} = 0$.  Throughout the paper, the letter $C$ will denote a generic positive constant.

Our convergence results, like many in the Markov chain literature, are described in terms of the \textit{Wasserstein distance}. For a pair of measures $\mu, \nu$ on a Polish space $(\Omega, d)$, let $\Pi(\mu,\nu)$ be the set of all couplings of $\mu$ and $\nu$. Then the Wasserstein distance between $\mu, \nu$ is given by

\be \label{EqDefWass}
W_{d}(\mu,\nu) = \inf_{\zeta \in \Pi(\mu,\nu)} \int_{x,y \in \Omega} d(x,y) \zeta(dx,dy).
\ee

The \textit{Total Variation} distance between two measures $\mu, \nu$ on a probability space $(\Omega, \mathcal{F}, \P)$ is given by
\be 
\| \mu - \nu \|_{\TV} = \sup_{A \in \mathcal{F}} | \mu(A) - \nu(A) |.
\ee 
This can be viewed as a special case of the Wasserstein distance, obtained by setting $d(x,y) = \textbf{1}_{x \neq y}$.

We study the convergence of Markov chains through the notions of mixing times, curvature and drift functions. Recall that the \textit{mixing time} of a Markov chain $\{X_{t}\}_{t \geq 0}$ with kernel $K$ and stationary distribution $\pi$ on state space $\Omega$ is given by 
\be 
\tmix = \sup_{X_{0} = x \in \Omega} \inf \{ t \geq 0 \, : \, \| \mathcal{L}(X_{t}) - \pi \|_{\TV} < \frac{1}{4} \}. 
\ee 

We follow \cite{Olli09} in defining the \textit{Ricci curvature} of a transition kernel $K$ on $(\Omega, d)$ by
\be
\kappa(x,y) = 1 - \frac{W_{d}(K(x, \cdot), K(y, \cdot))}{d(x,y)},
\ee
and the curvature of the entire chain is defined to be
\be
\kappa = \inf_{x,y \in \Omega} \kappa(x,y).
\ee

A positive curvature for a power $K^{s}$ of $K$ implies that $K$ has good mixing properties. For example, if the curvature of $K^{s}$ with respect to the metric $d(x,y) = \textbf{1}_{x \neq y}$ is $\kappa$, it is straightforward to check that the mixing time of $K$ is at most $\lceil \frac{s \, \log(4)}{\log( (1 - \kappa)^{-1})} \rceil \approx \frac{s}{\kappa} \log(4)$.
It is worth noting that in many cases of interest, it is sufficient to calculate $\kappa(x,y)$ for $d(x,y)$ `small'; see, \textit{e.g.}, Prop 19 of \cite{Olli09}. 

Following \cite{JoOl10}, we also define the \textit{eccentricity} of a point $x \in \Omega$ by:
\be
E(x) = \int_{\Omega} d(x,y) \pi(dy).
\ee

Finally, we say that a Markov chain $\{ X_{t} \}_{t \geq 0}$ with kernel $K$ satisfies a \textit{drift condition} with \textit{Lyapunov function} $V$ and constants $a,b,\ell$ if 
\be \label{EqLyapunovDef}
\E[V(X_{t+\ell}) | X_{t}] \leq (1 - a) V(X_{t}) + b.
\ee 
Most of the proofs are given in an Appendix.

\section{Technical Results} \label{SecConvBounds}

Throughout this section, we study pairs of Markov chains $\{ X_{t} \}_{t \geq 0}$, $\{ \tilde{X}_{t} \}_{t \geq 0}$ with transition kernels $K, \tilde{K}$ and stationary distributions $\pi, \tilde{\pi}$ on a Polish space $(\Omega, d)$. Throughout this section and the rest of the paper, $K$ will generally refer to a generic transition kernel that we would like to simulate from, and $\tilde{K}$ will generally refer to an approximation of $K$. 

\subsection{Convergence of Chains} \label{SubsecConvIndPoints}

We begin with a general mixing estimate that most of our bounds will follow from. First, an assumption:

\begin{assumption} [Drift-Like Condition] \label{AssDriftLike}
Let $K$ be a kernel on a Polish space $(\Omega,d)$ with stationary distribution $\pi$, let $\mathcal{X} \subset \Omega$ and let $p \in \mathcal{X}$. Let $\{ Z_{t} \}_{t \geq 0}$ be a Markov chain that evolves according to $K$. Say that $K$ satisfies a \textit{drift-like condition} with respect to $(\mathcal{X},p)$ if there exist constants $0 < \beta < 1$ and $\mathcal{B}, \mathcal{C}, \mathcal{D}, c_{p} < \infty$ and a positive function $L \, : \, \Omega \mapsto \mathbb{R}^{+}$ so that 
\be \label{IneqThreePart}
\E[L(Z_{t+1}) \vert Z_{t}] &\leq (1 - \beta) L(Z_{t}) + \mathcal{B} \\
\sup_{x \in \mathcal{X}} \E[L(Z_{t+1}) \vert Z_{t} = x, Z_{t+1} \notin \mathcal{X}] &\leq \mathcal{C}. \\
\E_{\pi}[L(X) \vert X \notin \mathcal{X}] &\leq \mathcal{D}.
\ee 
and 
\be 
d(x,p) \leq L(x) + c_{p} 
\ee 
for all $x \in \Omega$.
\end{assumption}

\begin{remark}
This assumption is generally easy to check using the same calculations that one uses to establish a drift condition in the usual sense (see inequality \eqref{EqLyapunovDef}). Indeed, as long as the Lyapunov function is not pathological, this assumption will be implied by the usual drift condition.
\end{remark}

\begin{lemma}\label{LemmaWassCont}
Assume that $K$ has eccentricity $E(x) < \infty$ with respect to $\pi$. Assume that there exists some set $\mathcal{X} \subset \Omega$ so that
\be  \label{IneqSimpleWassContraction}
\sup_{x,y \in \mathcal{X}} \frac{W_{d}( K(x,\cdot), K(y,\cdot))}{d(x,y)} \leq  (1 - \alpha)
\ee 
and that there exists some $0 < \delta < \infty$ so that
\be \label{IneqKernCloseWassCont}
\sup_{x \in \mathcal{X}} W_{d}( \tilde{K}(x,\cdot), K(x,\cdot))< \delta.
\ee 
Assume that $K$, $\tilde{K}$ satisfy assumption \ref{AssDriftLike} with respect to the same $(\mathcal{X},p)$, and with the same constants. Then, for any initial point $\tilde{X}_{0} = x \in \Omega$,
\be \label{IneqAbsCoupResTwo}
W_{d}(\mathcal{L}(\tilde{X}_{T}), &\pi ) \leq \frac{\delta}{\alpha} + (1-\alpha)^{T} E(x) + \left( 2 \frac{\mathcal{B} }{\beta} + (1 - \beta)^{T} (L(x) + \mathcal{D}) + 2 c_{p} \right) \pi(\mathcal{X}^{c})  \\
&+  2 \left( 1 - \P[\{ X_{t} \}_{t=1}^{T-1} \cup \{ \tilde{X}_{t} \}_{t =1}^{T-1} \subset \mathcal{X}] \right) \left(  \mathcal{C}  + \frac{\mathcal{B} }{\beta} + c_{p} \right).
\ee 
\end{lemma}

In the common situation that a contraction bound holds uniformly, this gives: 

\begin{cor} \label{CorWassCont}
Assume that $K$ has eccentricity $E(x) < \infty$ with respect to $\pi$. Assume that 
\be  \label{IneqSimpleWassContractionNice}
\sup_{x,y \in \Omega} \frac{W_{d}( K(x,\cdot), K(y,\cdot))}{d(x,y)} \leq  (1 - \alpha)
\ee 
and that inequality \eqref{IneqKernCloseWassCont} is satisfied for some $0 < \delta < \infty$.

Then, for any initial point $\tilde{X}_{0} = x \in \Omega$, 
\be \label{IneqAbsCoupResTwoNice}
W_{d}(\mathcal{L}(\tilde{X}_{T}), \pi ) \leq  \frac{\delta}{\alpha} + (1-\alpha)^{T} E(x)
\ee 
and 
\be \label{IneqAbsCoupResTwoStatNice}
W_{d}(\tilde{\pi}, \pi) \leq \frac{\delta}{\alpha}.
\ee 
\end{cor}

Taking powers of $K$ and noting that all kernels are contractive in Total Variation distance, Lemma \ref{LemmaWassCont} has the immediate corollary:

\begin{cor} \label{LemCoupIneqAdapt2}
Denote by $\tmix$, $\tilde{\tau}_{\mathrm{mix}}$ the mixing times of $K$ and $\tilde{K}$, and assume that
\be \label{IneqTvContAssumption}
\sup_{x \in \Omega} \| \tilde{K}(x,\cdot) - K(x,\cdot) \|_{\TV} < \delta
\ee 
for some $\delta > 0$. Then

\be \label{IneqAbsCoupResThree}
\| \mathcal{L}(\tilde{X}_{T}) -  \pi \|_{\TV} \leq \frac{4 \delta}{3} \tmix  + 4^{-\lfloor \frac{T}{\tmix} \rfloor}
\ee 
and 
\be \label{IneqAbsCoupResThreeStat}
\| \tilde{\pi} -  \pi \|_{\TV} \leq \frac{4 \delta}{3} \tmix.
\ee 
Under the additional assumption that $\delta < \frac{9}{128 \tmix}$,
\be \label{IneqAbsCoupResThreeMix}
\tilde{\tau}_{\mathrm{mix}} \leq 2 \tmix.
\ee
\end{cor}

\begin{remark} \label{RemarkCompFriel}
We mention that a result very similar to Corollaries \ref{CorWassCont}, \ref{LemCoupIneqAdapt2} is also immediately implied by Corollary 3.1 of \cite{Mitr05}; the constants are also very similar in our regime of interest. Although our Lemma \ref{LemmaWassCont} and the results in \cite{Mitr05} both imply the same result in this restricted setting, they have different emphases. Their result is based on linear algebra; ours is purely probabilistic. Our results also apply to chains that are not uniformly ergodic, as chains can have Wasserstein contraction (or Total Variation contraction on small sets) without being uniformly ergodic. Finally, their result only applies for convergence in Total Variation, while our results explicitly allow the use of many other Wasserstein metrics; this flexibility can lead to bounds that are effectively much sharper if the metric is chosen carefully. \par 
This last difference is most easily seen when the Markov chain $K$ satisfies inequality \eqref{IneqSimpleWassContraction} for some fixed $\alpha > 0$ throughout a non-compact state space, and for which the eccentricity satisfies $E(x) < \infty$ for each $x \in \Omega$ but $\sup_{x \in \Omega} E(x) = \infty$. These chains are generally geometrically ergodic but not uniformly ergodic. Lemma \ref{LemmaWassCont} can provide direct bounds on their finite-time bias while Corollary 3.1 of \cite{Mitr05} does not apply. This class of Markov chains contains many Markov chains of interest. For example, it includes the Metropolis-Hastings chain with proposal distribution $L(x,\cdot) = \mathcal{N}(x,\sigma_{1})$ and target stationary distribution $\pi = \mathcal{N}(\mu, \sigma_{2})$ as long as $\sigma_{1} < \sigma_{2}$. 
\end{remark}

\subsection{Convergence of Monte Carlo Estimates}
The main bounds in subsection \ref{SubsecConvIndPoints} are most useful when some power $K^{\ell}$ of the transition kernel $K$ of interest has positive curvature. This is a (strictly) weaker condition than uniform ergodicity of $K$, but it fails to hold for many Markov chains of statistical interest. In this section, we weaken the requirement of \textit{global} positive curvature for some power $K^{\ell}$ of $K$ to the requirement of \textit{local} positive curvature on a small set \textit{and} a drift condition. 

We weaken the requirements by showing that, if $K$ and $\tilde{K}$ both satisfy a drift condition \textit{and} $\tilde{K}$ is close to $K$ \textit{on a small set}, then
\begin{enumerate}
\item  The bias $\tilde{\pi}(f) - \pi(f)$ is small, and
\item  The Monte Carlo errors $| \frac{1}{T} \sum_{t=1}^{T} f(\tilde{X}_{t}) - \tilde{\pi}(f) |$ and $| \frac{1}{T} \sum_{t=1}^{T} f(X_{t}) - \pi(f) |$ are small with high probability.
\end{enumerate}
Thus, they allow us to estimate the total error of a Monte Carlo estimate obtained from $\tilde{K}$, based only on a drift condition and on a good approximation \textit{within a small set}. 

We begin by showing that the bias is small. To state our result, we recall  the definition of the \textit{trace} $\{X_{t}^{(S)}\}$ of a Markov chain $\{X_{t}\}_{t \geq 0}$ onto a set $S$. Let $T(S) = \{ t \geq 0 \, : \, X_{t} \in S \}$ be an ordered set. When $|T(S)| = \infty$, we define
\be 
\{X_{t}^{(S)} \}_{t \geq 0 } = \{ X_{t} \}_{t \in T(S)},
\ee 
again viewed as an ordered set. We have:

\begin{thm} [Bias for Geometrically Ergodic Chains] \label{BiasGeomErg}
Let $\{X_{t}\}_{t \geq 0}$, $\{\tilde{X}_{t} \}_{t \geq 0}$ be two Markov chains with kernel $K, \tilde{K}$. Assume that the chains satisfy
\be 
\E[V(X_{t+\ell}) | X_{t}] \leq (1-a)V(X_{t}) + b \\
\E[\tilde{V}(\tilde{X}_{t+\ell}) | \tilde{X}_{t}] \leq (1-\tilde{a})V(\tilde{X}_{t}) + \tilde{b} \\
\ee 
for some $\ell \in \mathbb{N}$ and some $0 < a, \tilde{a} \leq 1$ and some $0 \leq b, \tilde{b} < \infty$. Fix a constant $\max(\frac{4b}{a}, \frac{4\tilde{b}}{\tilde{a}}) < C < \infty$ and let $\mathcal{X} = \{x \in \Omega \, : \, V(x) < C, \, \tilde{V}(x) < C \}$. Assume that the trace of $\{ X_{t} \}_{t \geq 0}$ on $\mathcal{X}$ has mixing time $\tmix$ and that 
\be 
\max_{x \, : \, V(x) \leq C, \, \tilde{V}(x) \leq C} \| K(x,\cdot) - \tilde{K}(x,\cdot) \| \leq \delta.
\ee 
Then the stationary distributions $\pi, \tilde{\pi}$ of $K$ and $\tilde{K}$ satisfy
\be 
\| \pi - \tilde{\pi} \|_{\TV} \leq \frac{4 \delta}{3} \tmix + \frac{2b}{a C} +  \frac{2\tilde{b}}{\tilde{a} C}.
\ee 
\end{thm}

\begin{remark}
In applications, where $\tilde{K} = K_{n}$ is a Markov chain based on subsampling roughly $n$ of $N$ points, we expect roughly $\delta \sim \frac{\log(C)}{\sqrt{n}}$. Optimizing this leads to bias that goes to 0 a little more quickly than $O(\frac{\log(n)}{\sqrt{n}})$. This is comparable to the $O(\frac{1}{\sqrt{n}})$ rate expected for uniformly ergodic chains.
\end{remark}

The second point, that the  Monte Carlo errors $| \frac{1}{T} \sum_{t=1}^{T} f(\tilde{X}_{t}) - \tilde{\pi}(f) |$ and $| \frac{1}{T} \sum_{t=1}^{T} f(X_{t}) - \pi(f) |$ are small with high probability, follows from existing bounds. In particular, it follows by combining existing concentration inequalities for Markov chains (\textit{e.g.}  Theorem 2 of \cite{dedecker2014subgaussian} or various bounds in \cite{paulin2012concentration})  with the bound on the mixing time $\ttmix$ of $\tilde{K}$ on small sets given in Corollary \ref{LemCoupIneqAdapt2}.

The bounds in this section assume that $\tilde{K}$ satisfies a drift condition. We find that, when a drift condition can be proved for $K$ by hand, it is often not hard to check by hand that $\tilde{K}$ also satisfies a drift condition. However, the general question of when $\tilde{K}$ inherits a drift condition from $K$ seems to be quite difficult. We find that the following simple bound is often useful for `nice' chains $K$, $\tilde{K}$:  

\begin{lemma} [Drift and Minorization of Approximate Chains] \label{LemmDriftCondAustFrame}
Let $K$, $\tilde{K}$ be transition kernels associated with proposal kernel $L$ and acceptance functions $\alpha$, $\tilde{\alpha}$. Assume that $K$ satisfies inequality \eqref{EqLyapunovDef} Also assume that $K$, $\tilde{K}$ satisfy
\be 
\sup_{x \in \Omega} \| K(x,\cdot) - \tilde{K}(x,\cdot) \|_{\TV} &\leq \delta \\
\vert \alpha(x,y) - \tilde{\alpha}(x,y) \vert \leq  \delta f(x,y) &\leq \delta
\ee 
for some $f, \, \delta > 0$, where
\be \label{IneqStrongMultAssump2}
\int_{y} f(x,y) V(y) L(x,dy) &\leq V(x). 
\ee 
Then a chain $X_{t}$ evolving according to $\tilde{K}$ satisfies a drift condition of the form
\be 
\E[V(X_{t+1}) \vert X_{t} = x] \leq (1 - a + 2 \delta) V(x) + b.
\ee 
\end{lemma}

For further study of this question, see \textit{e.g.} \cite{rudolf2015perturbation, FHL13}.

\begin{remark} 
Lemma \ref{LemmDriftCondAustFrame} is our first result that specifies a relationship between the proposal kernels associated with the two chains $K$ and $\tilde{K}$. We force the two kernels to have the same proposal distribution, which may be surprising: it is well-known that a subsampling MCMC algorithm should \textit{not} use the same proposal kernel as the original MCMC dynamics (see Remark \ref{RemPropKernChoice} below for more on this choice).

This choice is not an accident. Perturbation bounds are only effective when the original kernel $K$ and the approximate kernel $\tilde{K}$ are `close' to each other, and in general the true MCMC dynamics are not `close' to the dynamics of subsampling MCMC. As a consequence, Lemma \ref{LemmDriftCondAustFrame} and the other bounds in this section often give very poor bounds when applied directly to these two chains. Instead, they will generally be applied to a collection of interpolating chains. This is discussed in further detail starting in Section \ref{SecCompChoice}, where our interpolating chains are defined. 
\end{remark}

\section{Bias-Variance Tradeoff and Applications to Austerity Framework} \label{SecAusterity}

In this section, we recall various `approximate' MCMC chains, introduce measures of computational and statistical efficiency, give a general tradeoff result for these measures, and apply the tradeoff result to certain examples.

\subsection{Approximate Metropolis-Hastings Algorithms} \label{SubsecApproxMhAlgs}

We consider a data set $\mathcal{X} = \{x_{1},\ldots,x_{N} \}$ and are interested in sampling from a posterior distribution  
\be \label{EqDefPostDist}
\pi(\theta) \equiv \pi(\theta \vert \{x_{i} \}_{i=1}^{N}) = p(\theta) \prod_{i=1}^{N} p(\theta \vert x_{i}).
\ee

One standard tool is the Metropolis-Hastings algorithm. Fix a reversible transition kernel $L$ on state space $\Omega$. We recall the following form of the Metropolis-Hastings algorithm from \cite{BDC14}:

\begin{algorithm}[!ht] \caption{Metropolis-Hastings Algorithm}
\label{AlgGenMH}
\begin{algorithmic} 
\STATE Initialize $X_{1} = x$.
\FOR  {$t = 1$ to $T$}
\STATE Generate $\ell_{t+1} \sim L(X_{t}, \cdot)$, $u_{t} \sim \mathbb{U}[0,1]$.
\STATE Set $\psi(u_{t},X_{t},\ell_{t+1}) = \frac{1}{N} \log(u_{t} \frac{p(X_{t}) L(X_{t},\ell_{t+1})}{ p(\ell_{t+1}) L(\ell_{t+1},X_{t})})$, $\Lambda(X_{t},\ell_{t+1}) = \frac{1}{N} \sum_{i=1}^{N} \log(\frac{p(x_{i} | \ell_{t+1})}{p(x_{i} | X_{t})})$.
\IF {$\Lambda(X_{t},\ell_{t+1}) > \psi(u_{t},X_{t},\ell_{t+1})$}
\STATE  Set $X_{t+1} = \ell_{t+1}$.
 \ELSE
\STATE
 Set $X_{t+1} = X_{t}$.
\ENDIF
\ENDFOR
\end{algorithmic}
\end{algorithm}

Throughout the rest of the paper, the kernel $\KMH$ will refer to the kernel associated with Algorithm \ref{AlgGenMH}. The goal of all approximate algorithms \cite{BDC14, AFEB14, quiroz2014speeding, rudolf2015perturbation} is to make the decision in step 5 of this algorithm without doing the complete computation in step 4. 

Algorithm \ref{AlgGenAustFrame}, first suggested in \cite{BDC14}, is a representative approximate MCMC algorithm. The associated constants are
\be 
f_{t}^{\ast} &= \frac{t-1}{N} \\
C_{\theta,\theta'} &= \max_{1 \leq i \leq N} | \log( p(x_{i} | \theta')) - \log( p(x_{i} | \theta)) | \\ 
c_{t} &= C_{\theta,\theta'} \sqrt{\frac{2(1-f_{t}^{\ast}) \log(2 / \delta_{t})}{t}}
\ee 
and any number $0 < \gamma < \infty$.

\begin{algorithm}[!ht] \caption{Subsampling MCMC}
\label{AlgGenAustFrame}
\begin{algorithmic} 
\STATE Initialize $X_{1} = x$.
\FOR {$t = 1$ to $T$}
\STATE Generate $\ell_{t+1} \sim L(X_{t}, \cdot)$, $u_{t} \sim \mathbb{U}[0,1]$.
\STATE Set $\psi(u_{t},X_{t},\ell_{t+1}) = \frac{1}{N} \log(u_{t} \frac{p(X_{t}) L(X_{t},\ell_{t+1})}{ p(\ell_{t+1}) L(\ell_{t+1},X_{t})})$.
\STATE Set $s=0$, $s_{\mathrm{look}} = 0$, $\Lambda^{\ast} = 0$, $\mathcal{X}^{\ast} = \emptyset$, $b =1$, $DONE = FALSE$.
\WHILE{$DONE == FALSE$}
\STATE     Sample $x_{s+1}^{\ast},\ldots,x_{b}^{\ast} \sim_{\text{w/o repl}} \mathcal{X} \backslash \mathcal{X}^{\ast}$.
\STATE     Set $\mathcal{X}^{\ast} = \mathcal{X}^{\ast} \cup \{x_{s+1}^{\ast},\ldots,x_{b}^{\ast} \}$, $\Lambda^{\ast} = \frac{1}{b}(t \Lambda^{\ast} + \sum_{i=s+1}^{b} \log (\frac{p(x_{i}^{\ast} | \ell_{t+1})}{p(x_{i}^{\ast} | X_{t})}) )$, and $s = b$.
\STATE     Set $c = 2C_{X_{t},\ell_{t+1}} \sqrt{\frac{(1 - f_{s}^{\ast}) \log(\frac{2}{\delta_{s_{\mathrm{look}}}})}{2 s}}$, $s_{\mathrm{look}} = s_{\mathrm{look}} + 1$, $b = \min(N, \lceil \gamma s \rceil)$.
\IF  {$|\Lambda^{\ast} - \psi(u,X_{t},\ell_{t+1}) | \geq c$ or $b > N$} 
\STATE Set $DONE = TRUE$. 
\ENDIF
\ENDWHILE
\IF {$\Lambda^{\ast} > \psi(u_{t},X_{t},\ell_{t+1})$}
\STATE Set $X_{t+1} = \ell_{t+1}$.
 \ELSE
\STATE Set $X_{t+1} = X_{t}$.
\ENDIF
\ENDFOR
\end{algorithmic}
\end{algorithm}

Algorithm \ref{AlgGenMhNarrow} provides a very simple subsampling algorithm, with parameter $1 \leq n \leq N$ representing the amount of subsampling. For the remainder of the paper, $\KS{n}$ will refer to the kernel associated with Algorithm \ref{AlgGenMhNarrow}.

\begin{algorithm}[!ht] \caption{Simple Approximate Metropolis-Hastings Algorithm; Narrow Tails}
\label{AlgGenMhNarrow}
\begin{algorithmic} 
\STATE Initialize $X_{1} = x$.
\FOR{ $t = 1$ to $T$}

\STATE Generate $\ell_{t+1} \sim L(X_{t}, \cdot)$ and $u_{t} \sim \mathbb{U}[0,1]$. 
\STATE sample $x_{1}^{\ast},\ldots,x_{n}^{\ast}$ without replacement from $\mathcal{X}$, set $\psi(u_{t},X_{t},\ell_{t+1}) = \frac{1}{N} \log(u_{t} \frac{p(X_{t}) L(X_{t},\ell_{t+1})}{ p(\ell_{t+1}) L(\ell_{t+1},X_{t})})$, and set $\Lambda^{\ast}(X_{t},\ell_{t+1}) = \frac{1}{n} \sum_{i=1}^{n} \log(\frac{p(x_{i}^{\ast} | \ell_{t+1})}{p(x_{i}^{\ast} | X_{t})})$.
\IF {$\Lambda^{\ast}(X_{t},\ell_{t+1}) > \psi(u_{t},X_{t},\ell_{t+1})$}
\STATE  Set $X_{t+1} = \ell_{t+1}$.
 \ELSE
\STATE
 Set $X_{t+1} = X_{t}$.
\ENDIF
\ENDFOR
\end{algorithmic}
\end{algorithm}

\begin{remark} \label{RemPropKernChoice}

All of the algorithms stated in this section, and in later sections, are defined for any reversible kernel $L$. In practice, we are most interested in proposal kernels of the form
\be \label{EqFormPropDist}
L(x,y) \equiv L_{\sigma}(x,y) = f_{\sigma}(x-y),
\ee 
where $f_{\sigma}$ is a density with variance $\sigma$. Both in practice and all our theorems, the choice of $\sigma$ will be adapted to the variance of the stationary distribution of the underlying Markov chain - the variance of the proposal distribution should increase with the variance of the stationary distribution. In particular,  a subsampling MCMC algorithm should generally have a proposal distribution $f_{\sigma}$ with a larger variance than the optimal variance of the original Metropolis-Hastings dynamics. See \textit{e.g.} \cite{rosenthal2011optimal} for a survey on the optimization of proposal distributions.

\end{remark}

\subsection{Measures of Computational and Statistical Efficiency} \label{SubsecMeasCompStatEff}

We introduce measures of \textit{computational efficiency} for approximate MCMC, in the same tradition as earlier work such as \cite{sherlock2014efficiency, doucet2015efficient, bornn2014use, quiroz2014speeding}.

Recall that the goal of MCMC is often to estimate the mean $\pi(f)$ of some function $f$ with respect to a distribution $\pi$ of interest. This estimator is generally of the form
\be \label{EqMcmcEstimator}
\hat{\pi}_{T}(f) = \frac{1}{T+1} \sum_{t=0}^{T} f(X_{t}).
\ee
For an estimator of the form \eqref{EqMcmcEstimator} started at point $X_{0} = x$ and driven by the `approximate' Markov kernel $\tilde{K}$ with stationary distribution $\tilde{\pi}$, we denote by $v_{T}(f,\tilde{K},x)$ the mean-squared error of \eqref{EqMcmcEstimator}:
\be 
v_{T}(f,\tilde{K},x) = \E_{x}[ (\hat{\pi}_{T}(f) - \pi(f))^{2}].
\ee 

We associate with every Markov kernel $\tilde{K}$ the \textit{computational cost} $c(\tilde{K})$ of taking a single step. In our setting, this will be the average number of likelihoods that must be evaluated at every step; for example, we assign to the kernel $K_{n}^{(\mathrm{simple})}$ defined in Algorithm \ref{AlgGenMhNarrow} the cost $c(K_{n}^{(\mathrm{simple})}) = n$. 

We then assume that we have a limited \textit{total computational budget} $M$ and define the \textit{standardized error} of an algorithm $\tilde{K}$ to be
\be 
\mathcal{E}_{M}(f,\tilde{K},x) = v_{\lfloor \frac{M}{c(\tilde{K})} \rfloor}(f,\tilde{K},x),
\ee 
 the smallest error obtainable by running $\tilde{K}$ for $\lfloor \frac{M}{c(\tilde{K})} \rfloor$ steps.

$\mathcal{E}_{M}(f,\tilde{K},x)$ is difficult to compute, and so for the remainder of the paper we replace it with a pessimistic estimate. Let $\ttmix$ be the mixing time of $\tilde{K}$. Let $\{X_{t}\}_{t \geq 0}$, $\{Y_{t}\}_{t \geq 0}$ be Markov chains evolving according to $\tilde{K}$ and started at points $X_{0} = x \in \Omega$ and $\mathcal{L}(Y_{0}) = \tilde{\pi}$ respectively. Couple these chains to satisfy 
\be \label{EqOptimalCouplingDef}
\P[X_{(k+1)\ttmix} = Y_{(k+1)\ttmix} | X_{k \ttmix} \neq Y_{k\ttmix} ]  \geq \frac{3}{4}
\ee 
for all integers $k \geq 0$ and 
\be 
\{ X_{t} = Y_{t} \} \Rightarrow \{ X_{s} = Y_{s} \}
\ee 
for all pairs of integers $s > t \geq 0$. The fact that at least one such coupling exists follows from the definition of the mixing time. We note that, for all $f \, : \, \Omega \mapsto [0,1]$,
\be 
v_{T}(f,\tilde{K},x) = \E_{x}[ (\hat{\pi}_{T}(f) - \pi(f))^{2}] &\leq 3 ( \E_{x}[ (\frac{1}{T+1} \sum_{t=0}^{T} (f(X_{t}) - f(Y_{t})))^{2}] + \E_{x}[(\frac{1}{T+1} \sum_{t=0}^{T} f(Y_{t}) - \tilde{\pi}(f))^{2}] \\
&+ (\tilde{\pi}(f) - \pi(f))^{2}) \\
\ee 
We then define 
\be 
\mathcal{E}_{M}^{(1)}(\tilde{K},f) &=  \sup_{x \in \Omega}  \E_{x}[ (\frac{1}{T+1} \sum_{t=0}^{T} (f(X_{t}) - f(Y_{t})))^{2}]   \\
\mathcal{E}_{M}^{(2)}(\tilde{K},f) &=  \E[(\frac{1}{T+1} \sum_{t=0}^{T} f(Y_{t}) - \tilde{\pi}(f))^{2}] \\
\mathcal{E}_{M}^{(3)}(\tilde{K},f) &=  (\tilde{\pi}(f) - \pi(f))^{2} \\
\ee 
and set
\be \label{EqTotalError} 
\mathcal{E}_{M}(\tilde{K},f) = (\mathcal{E}_{M}^{(1)}(\tilde{K},f) + \mathcal{E}_{M}^{(2)}(\tilde{K},f) + \mathcal{E}_{M}^{(3)}(\tilde{K},f)).
\ee 
Note that the three error terms $\mathcal{E}_{M}^{(1)}$, $\mathcal{E}_{M}^{(2)}$ and $\mathcal{E}_{M}^{(3)}$ represent, respectively, 
\begin{enumerate}
\item Bias due to the chain not starting at stationarity (so-called `burn-in' bias). By inequality \eqref{EqOptimalCouplingDef}, this generally decays like 
\be \label{IneqBiasBasic}
\mathcal{E}_{M}^{(1)} = O(\frac{\ttmix^{2}}{T^{2}}).
\ee
\item The asymptotic variance of any MCMC estimate coming from $\tilde{K}$. By Equation 11.4 of \cite{aldous-fill-2014} and following discussion, this generally decays like 
\be 
\mathcal{E}_{M}^{(2)} = O(\frac{\ttmix}{T} \mathrm{Var}_{\tilde{\pi}}(f)). 
\ee
\item The asymptotic bias due to the fact that $\tilde{K}$ and $K$ don't have the same stationary measure. By Corollary \ref{LemCoupIneqAdapt2},  this generally decays like 
\be 
\mathcal{E}_{M}^{(3)} = O(\delta^{2} \ttmix^{2}), 
\ee 
where $\delta = \sup_{x \in \Omega} \| \tilde{K}(x,\cdot) - K(x,\cdot) \|_{\TV}$.
\end{enumerate}

Let $\mathbb{K}$ denote all kernels within some parameterized family (for example, all kernels defined by Algorithm \ref{AlgGenMhNarrow} as the parameter $n$ is allowed to vary from $1$ to $N$). In our framework, the main questions about the efficiency of approximate MCMC algorithms may be written as:

\begin{enumerate} \label{EnumThreeQuestions}
\item For a given computational budget $M$ and function $f$, what is $\tilde{K}^{\ast} = \mathrm{argmin}_{\tilde{K} \in \mathbb{K}} (\mathcal{E}_{M}(\tilde{K},f))$? 
\item How much smaller is $\mathcal{E}_{M}(\tilde{K}^{\ast},f)$ than the error $\mathcal{E}_{M}(K,f)$ of the reference kernel $K$?
\item We also add: are perturbation-based analyses sufficient to find good answers to these questions?
\end{enumerate}

\subsection{General Tradeoff Theorem} \label{SecTradeoffChoice}

In this section, we begin to answer these questions by giving a simple `tradeoff' theorem for approximate MCMC algorithms. This will give bounds on the possible values of the `best' approximate chain $K^{\ast}$. In particular, it will give sufficient conditions under which approximate MCMC algorithms do better than the underlying MCMC algorithm (that is, $K^{\ast} \neq K$). The theorem has many assumptions, and we will spend the rest of the paper giving conditions under which these assumptions hold.

Before giving the result, we state some simple technical lemmas:

\begin{lemma}[Burn-In Error] \label{LemmaBurninBound}
Fix a kernel $K$ with computational cost $c(K)$ and mixing time $\tmix$. Then for all $f \, : \, \Omega \mapsto [0,1]$, the error $\mathcal{E}_{M}^{(1)}(K)$ satisfies
\be 
\mathcal{E}_{M}^{(1)}(K,f) \leq 2 \tmix^{2} \big\lfloor \frac{M}{c(K)} \big\rfloor^{-2}.
\ee 
\end{lemma}

\begin{lemma}[Asymptotic Variance] \label{LemmaVarianceBound}
Fix a kernel $K$ with computational cost $c(K)$ and mixing time $\tmix$.  Then for all $f \, : \, \Omega \mapsto [0,1]$, the error $\mathcal{E}_{M}^{(2)}(K,f)$ satisfies
\be 
\mathcal{E}_{M}^{(2)}(K,f ) \leq 2 \tmix \mathrm{Var}_{\pi}(f)  \big\lfloor \frac{M}{c(K)} \big\rfloor^{-1}.
\ee 
\end{lemma}
\begin{proof}
This follows immediately from  Lemma 4.23 and Proposition 4.29 of \cite{aldous-fill-2014}.
\end{proof}

\begin{lemma}[Asymptotic Bias] \label{LemmaBiasBound}
Fix a pair of kernels $K$, $\tilde{K}$ with mixing times $\tmix$, $\ttmix$. Assume that 
\be 
\sup_{x \in \Omega} \| K(x,\cdot) - \tilde{K}(x,\cdot) \|_{\TV} \leq \delta 
\ee  
for some $0 < \delta < 1$. Then for all $f \, : \, \Omega \mapsto [0,1]$, the error $\mathcal{E}_{M}^{(3)}(K,f)$ satisfies
\be 
\mathcal{E}_{M}^{(3)}(K,f ) \leq \frac{16 \delta^{2}}{9} \min(\tmix, \ttmix)^{2}.
\ee 
\end{lemma}
\begin{proof}
This follows immediately from Corollary \ref{LemCoupIneqAdapt2}.
\end{proof}

We now give our main result. For reasons discussed in Section \ref{SecCompChoice}, the result uses two familiees of approximate kernels:

\begin{thm} [Convergence Rates for Approximate MCMC] \label{ThmTradeoff}
Let $K$ be a kernel with stationary distribution $\pi$ and mixing time $\tmix$. 

Fix  $N \in \mathbb{N}$. For $1 \leq n \leq N$ and $ 0 < c_{1}  < \infty$, let $\tilde{K}^{(n)}$ be an approximate kernel with cost $c(\tilde{K}^{(n)}) = n^{1+c_{1}}$, and let $\ttmix^{(n)}$, $\tilde{\pi}^{(n)}$ be the stationary distribution and mixing time of $\tilde{K}^{(n)}$. Also let $K^{(n)}$ be a reference kernel associated with parameter $n$, with stationary measure $\pi^{(n)}$ and mixing time $\tmix^{(n)}$. 

Assume that there exists $0 < C, C' < \infty$ so that 
\be \label{IneqScalingAssumptionTradeoff1}
C \tmix \leq \min_{1 \leq n \leq N} \tmix^{(n)} \leq \max_{1 \leq n \leq N} \tmix^{(n)} \leq C' \tmix.
\ee 
Fix a function $f \, : \, \Omega \mapsto [0,1]$, and assume that there exists $0 < c_{2}, C'', C''' < \infty$ so that 
\be \label{IneqScalingAssumptionTradeoff2}
\mathrm{Var}_{\pi^{(n)}}[f] \leq \frac{C''}{n},
\ee 

\be \label{IneqPerturbationAssumptionTradeoff}
\sup_{x \in \Omega} \| \tilde{K}^{(n)}( x,\cdot) - K^{(n)}(c,\cdot) \|_{\TV} \leq C''' \, n^{-c_{2}},
\ee 

and

\be  \label{IneqBiasAssumptionTradeoff}
| \pi(f) - \pi^{(n)}(f) | \leq C''' n^{-c_{2}}.
\ee

Fix a computational budget $M$. Then there exists a constant $n_{0} = n_{0}(C', C''', c_{2}, \tmix)$ so that for all $n_{0} \leq n \leq N$,
\be \label{IneqTradeoffInit}
\mathcal{E}_{M}(\tilde{K}^{(n)}, f) \leq 2 (C' \tmix)^{2} \big\lfloor \frac{M}{n^{1+c_{1}}} \big\rfloor^{-2} + 2 C' C'' \frac{\tmix}{n} \big\lfloor \frac{M}{n^{1+c_{1}}} \big\rfloor^{-1} + \frac{16}{9} (C' \tmix)^{2}  n^{-2 c_{2}} + C''' n^{-c_{2}},
\ee
where the error $\mathcal{E}_{M}$ given in equation \eqref{EqTotalError} is defined relative to the true kernel $K$.

Thus, if $n_{0} \leq M^{\frac{1}{2+c_{1}}}$, 

\be \label{IneqTradeoffMain}
\min_{1 \leq n \leq N} (\mathcal{E}_{M}(\tilde{K}^{(n)} f)) &\leq 2 (C' \tmix)^{2} \big\lfloor M^{\frac{1}{2 + c_{1}}} -1 \big\rfloor^{-2} + 2 C' C'' \tmix M^{-\frac{1}{2 + c_{1}}} \big\lfloor M^{\frac{1}{2 + c_{1}}} -1 \big\rfloor^{-1} \\
&+ (C''' + \frac{16}{9} (C' \tmix)^{2})  (M^{-\frac{2 c_{2}}{2 + c_{1}}} + 2).\\
\ee
\end{thm}

\begin{remark}
We explain how Theorem \ref{ThmTradeoff} can be applied. For sufficiently nice Markov chains, we expect the constants $C,C',C'',C'''$ in the assumptions in Theorem \ref{ThmTradeoff} to be uniform in $N$, or change only like $\log(N)$. The constants $c_{1},c_{2}$ are more complicated, and will be analyzed in later sections. For small $N$, the bounds in Section \ref{SecJust} (and particularly Lemma \ref{LemmaSmallPerturbation}) gives conditions under which the assumptions  of Theorem \ref{ThmTradeoff} hold with $c_{1} < 1 + \epsilon$ for all $\epsilon > 0$ and with $c_{2} = \frac{1}{2}$. On the other hand, Lemma \ref{LemmaKernelLimits} says that they hold with $c_{1} = o(1)$, $c_{2}^{-1} = o(1)$ as $N$ goes to infinity for any fixed $M$. In the regime where  Lemma \ref{LemmaKernelLimits} applies, this indicates that the optimal number of subsamples $n^{\ast} = n A$ to use per step satisfies $n^{\ast} = O(\sqrt{M})$.   When the bounds in  Section \ref{SecJust} hold, this indicates that the optimal number of subsamples satisfies $n^{\ast} = O(M^{\frac{2}{3}} \mathrm{polylog}(M))$. 

In both situations, the decay rate obtained in Theorem \ref{ThmTradeoff} cannot be obtained by using a subsample of size proportional to $N$. In particular, due to the term $\mathcal{E}^{(1)}$, it is impossible to get a convergence bound better than $O( \lfloor \frac{M}{n} \rfloor^{-2})$. Thus, when Lemma \ref{LemmaKernelLimits} applies, we find that $n^{\ast} \approx \sqrt{M}$. In either case, we have $n^{\ast} \ll \min(M,N)$ in many situations. This provides a partial answer to the first part of question \ref{EnumThreeQuestions}, showing that under our assumptions
\be 
\tilde{K}^{\ast}  \neq K
\ee 
for $N$ sufficiently large.

We believe that, in general, $n^{\ast} \approx \sqrt{M}$ in both regimes. However, our existing bounds sometimes only allow us to show that $n^{\ast} \ll \min(M,N)$. 
\end{remark}

\subsection{ Choice of Comparison Chain} \label{SecCompChoice}

To use Theorem \ref{ThmTradeoff}, we need candidate kernels for $\tilde{K}^{(n)}$ and $K^{(n)}$. Before giving these candidates later in this section, it is worth explaining why Theorem \ref{ThmTradeoff} can't be applied directly to subsampling MCMC algorithms such as  Algorithm \ref{AlgGenMhNarrow}. The problem is simple: when the subsample size $n$ is small, the dynamics of Algorithm \ref{AlgGenMhNarrow} are not actually a small perturbation of the true dynamics $K$ in the sense required by the bounds in Section \ref{SecConvBounds}. This means that it is hopeless to try to apply perturbation bounds directly.

This problem is not limited to Algorithm \ref{AlgGenMhNarrow}. For example, a similar problem occurs with Algorithm \ref{AlgGenAustFrame}. Roughly speaking, the number of points $b$ that must be sampled per step must generally be on the order of $N$ as $N$ goes to infinity if the error $\| \tilde{K}(x,\cdot) - K(x,\cdot) \|_{\TV}$ is to be small. Our observation is also not new; the requirement that $b \approx N$ is discussed in Section 3.2 of \cite{BDC14} in the context of upper bounds on $b$, but it is clear that the upper bounds are often tight (see \textit{e.g.} the anti-concentration bounds in \cite{li2001gaussian} for a way to make this rigorous). \cite{quiroz2014speeding} discusses a similar phenomenon for other subsampling MCMC chains. Thus, existing approximate MCMC algorithms for which perturbation bounds apply often can't even take a single step in the regime $N \gg M$, and the `burn-in' error $\mathcal{E}^{(1)}$ is $O(1)$ in the regime $N \approx M$ (see equation \eqref{IneqBiasBasic}). 

This problem means that, if we wish to use the powerful tool of perturbation theory to analyze approximate transition kernels outside of the regime $M \gg N$, it is necessary to construct a sequence of interpolating kernels for the purpose of comparison.

Our candidates for $K^{(n)}$ and $\tilde{K}^{(n)}$ will be `wide-tailed' versions of the algorithms in Section \ref{SubsecApproxMhAlgs}. The first is Algorithm \ref{AlgMhWide}; we will write $\KMHW{n}$ for the kernel of this algorithm. Algorithm \ref{AlgMhWide} is a wide-tailed version of Algorithm \ref{AlgGenMH} - it differs only in that the stationary distribution is `flattened' by the change in the fourth line of the algorithm. See Lemma \ref{LemmaWideMhProps} for some relationships between Algorithms \ref{AlgMhWide} and \ref{AlgGenMH}.  

\begin{algorithm}[!ht] \caption{Metropolis-Hastings Algorithm; Wide Tails}
\label{AlgMhWide}
\begin{algorithmic} 
\STATE Initialize $X_{1} = x$.
\FOR  {$t = 1$ to $T$}
\STATE Generate $\ell_{t+1} \sim L(X_{t}, \cdot)$, $u_{t} \sim \mathbb{U}[0,1]$.
\STATE Set $\psi(u_{t},X_{t},\ell_{t+1}) = \frac{1}{n} \log(u_{t} \frac{p(X_{t}) L(X_{t},\ell_{t+1})}{ p(\ell_{t+1}) L(\ell_{t+1},X_{t})})$, $\Lambda(X_{t},\ell_{t+1}) = \frac{1}{N} \sum_{i=1}^{N} \log(\frac{p(x_{i} | \ell_{t+1})}{p(x_{i} | X_{t})})$.
\IF {$\Lambda(X_{t},\ell_{t+1}) > \psi(u_{t},X_{t},\ell_{t+1})$}
\STATE  set $X_{t+1} = \ell_{t+1}$.
 \ELSE
\STATE
 set $X_{t+1} = X_{t}$.
\ENDIF
\ENDFOR
\end{algorithmic}
\end{algorithm}

The analogous `wide-tailed' version of Algorithm \ref{AlgGenMhNarrow} is given by Algorithm \ref{AlgGenMhWide} below.  We write $\KSW{C}{n}$ for  the kernel associated with Algorithm \ref{AlgGenMhWide}. It is generally not hard to write down `wide-tailed' versions of other approximate-MCMC algorithms, such as Algorithm \ref{AlgGenAustFrame}.

\begin{algorithm}[!ht] \caption{Simpler Approximate Metropolis-Hastings Algorithm; Wide Tails}
\label{AlgGenMhWide}
\begin{algorithmic} 
\STATE Initialize $X_{1} = x$.
\FOR{ $t = 1$ to $T$}

\STATE Generate $\ell_{t+1} \sim L(X_{t}, \cdot)$ and $u_{t} \sim \mathbb{U}[0,1]$. 
\STATE sample $x_{1}^{\ast},\ldots,x_{C n}^{\ast}$ without replacement from $\mathcal{X}$, set $\psi(u_{t},X_{t},\ell_{t+1}) = \frac{1}{n} \log(u_{t} \frac{p(X_{t}) L(X_{t},\ell_{t+1})}{ p(\ell_{t+1}) L(\ell_{t+1},X_{t})})$, and set $\Lambda^{\ast}(X_{t},\ell_{t+1}) = \frac{1}{C n} \sum_{i=1}^{C n} \log(\frac{p(x_{i}^{\ast} | \ell_{t+1})}{p(x_{i}^{\ast} | X_{t})})$.
\IF {$\Lambda^{\ast}(X_{t},\ell_{t+1}) > \psi(u_{t},X_{t},\ell_{t+1})$}
\STATE  set $X_{t+1} = \ell_{t+1}$.
 \ELSE
\STATE
 set $X_{t+1} = X_{t}$.
\ENDIF
\ENDFOR
\end{algorithmic}
\end{algorithm}

\subsection{Justification for Tradeoff Bounds} \label{SecJust}

We give bounds on the main quantities that appear in the assumptions of Theorem \ref{ThmTradeoff}, providing bounds of the form \eqref{IneqScalingAssumptionTradeoff1}, \eqref{IneqScalingAssumptionTradeoff2}, \eqref{IneqPerturbationAssumptionTradeoff} and  \eqref{IneqBiasAssumptionTradeoff}. In particular, we show that the approximate wide-tailed Metropolis-Hastings kernel  $\KSW{A}{n}$ is often a small perturbation of the wide-tailed Metropolis-Hastings kernel $\KMHW{n}$, and we show that the wide-tailed Metropolis-Hastings kernel $\KMHW{n}$ can inherit good properties from the usual Metropolis-Hastings kernel $\KMH$ (even though it is generally not a small perturbation).

We first show that $\KSW{A}{n}$ is close to $\KMHW{n}$, which can be used to obtain a bound of the form \eqref{IneqPerturbationAssumptionTradeoff}:

\begin{lemma} \label{LemmaSmallPerturbation} 
Fix $N, A,n \in \mathbb{N}$ with $1 \leq A \leq N$ and $1 \leq n \leq \frac{N}{A}$. Let $\KSW{A}{n}$ and $\KMHW{n}$ be the kernels associated with Algorithms \ref{AlgGenMhWide} and \ref{AlgMhWide} respectively. Assume that 
\be
\max_{1 \leq i \leq n}  \sup_{\theta \in \Omega} | \log(p(\theta  | x_{i} )) | \leq C \\
\ee
for some $0 < C < \infty$. Then for $A = A(n) \geq 4 C^{2} n \log(n)$, 
\be 
\sup_{x \in \Omega} \| \KSW{A}{n}(x,\cdot) - \KMHW{n}(x,\cdot) \|_{\TV} \leq \frac{2}{\sqrt{n}}.
\ee 
\end{lemma}

\begin{remark}
This gives a bound of the form \eqref{IneqPerturbationAssumptionTradeoff}.
\end{remark}

Our next bound will let us show that $\KMHW{n}$ inherits good properties From $\KMH$, allowing us to prove bounds of the form \eqref{IneqScalingAssumptionTradeoff1}. We need some additional technical assumptions:

\begin{assumption} [Convergence of Log-Concave Distributions] \label{DefConvDist}

Let $\{ \phi_{n} \}_{n \in \mathbb{N}}$ be a sequence of density functions on $\mathbb{R}^{d}$ and $\{ s_{n} \}_{n \in \mathbb{N}}$ a sequence of positive numbers. Assume:
\begin{enumerate}
\item Each density function $\phi_{n}$ is log-concave.
\item $\lim_{n \rightarrow \infty} s_{n} = \infty$.
\item There exists a density $\phi$ so that $\phi_{n}(s_{n} \, \cdot) \rightarrow_{n} \phi(\cdot)$ in distribution.
\item $\phi$ is a continuous function on its support.
\item For any compact set $\mathcal{X} \subset \Theta$, $c(\mathcal{X}) \equiv \inf_{x \in \mathcal{X}} \phi(x) > 0$.
\end{enumerate}
\end{assumption}

\begin{remark}
This assumption will be applied to the stationary distribution of the kernels $\KMHW{n}$, which are of the form $\phi_{n}(\cdot) = p(\cdot | (x_{1},\ldots,x_{n}))^{\frac{1}{m(n)}}$, where $\{ m(n) \}_{ n \in \mathbb{N}}$ is some sequence of integers.  

This assumption holds with probability 1 for many sequences of this form. The log-concavity condition is strong but easy to check (it holds, \textit{e.g.}, if the model $p(\theta | x)$ and the prior $p(\theta)$ are both Gaussian). The remainder of the assumption follows from an application of the Bernstein-Von Mises theorem, whenever the Bernstein-Von Mises theorem holds and $\{ m(n) \}_{n \in \mathbb{N}}$ satisfies $\lim_{n \rightarrow \infty} \frac{m(n)}{n} = 0$. Note that stating our assumption in this way means that it applies to models with posteriors that are not asymptotically Gaussian.
\end{remark}

%and

%\begin{assumption} [Convergence of Log-Concave Posteriors] \label{DefConvPost}
%Let $p(\theta)$, $p(x | \theta)$ be a statistical model with parameter space $\Theta$, where $\Theta$ is a vector space. Fix $\theta_{0} \in \Theta$ and let $x_{1},x_{2},\ldots$ be an i.i.d. sequence drawn from $p( \cdot | \theta_{0})$. 

%Assume that the sequence of distributions $\{ \phi_{n} \}_{ n \in \mathbb{N}} \equiv \{ p(\cdot | (x_{1},x_{2},\ldots,x_{n})) \}_{n \in \mathbb{N}}$ satisfies assumption \ref{DefConvDist} with probability 1.
%\end{assumption}

%\begin{remark}
%This assumption holds for many examples of interest. The log-concavity condition is easy to check (and holds, \textit{e.g.}, for Gaussian data with Gaussian priors), while the remainder of the assumption is a consequence of the Bernstein-Von Mises theorem when that result holds. Note that stating our assumption in this way means that it applies to models with posteriors that are not asymptotically Gaussian.
%\end{remark}

We can now state our convergence result (we suspect that this result is known, but could not find it in the literature):

\begin{theorem} [Convergence of Markov Chain Paths] \label{ThmUnifConv}
Let $\{ \phi_{n} \}_{n \in \mathbb{N}}$, $\{ s_{n} \}_{n \in \mathbb{N}}$ and $\phi$ satisfy assumption \ref{DefConvDist}. Fix a symmetric proposal kernel $L$ on $\Theta$ and define the kernel $L_{n}(x,A) = L(s_{n} x, s_{n} A)$ for all $x \in \Theta$, $A \subset \Theta$. Let $K$ be the Metropolis-Hastings kernel with proposal distribution $L$ and target distribution $\phi$. Let $K_{n}$ be the Metropolis-Hastings kernel with proposal distribution $L_{n}$ and target distribution $\phi_{n}$. Finally, fix $T \in \mathbb{N}$ and let $\{X_{t}\}_{t=0}^{T}$, $\{Y_{t}\}_{t=0}^{T}$ be Markov chains drawn from kernels $K_{n}$ and $K$ respectively, started at points $\frac{x}{s_{n}}$ and $x$ respectively. Then
\be 
\lim_{N_{0} \rightarrow \infty} \sup_{N_{0} \leq n} |\P[(s_{n} X_{0},\ldots,s_{n} X_{T}) \in A] - \P[(Y_{0},\ldots,Y_{0}) \in A] | = 0
\ee 
for any $A \in \Theta^{T+1}$.
\end{theorem}

\begin{cor}
Set notation as in Theorem \ref{ThmUnifConv}. Let $\tau_{n}$ be the mixing time of $K_{n}$ and let $\tau$ be the mixing time of $K$. For fixed $A \subset \Theta$, let $\tau_{n}(A)$ be the mixing time of the trace of $K_{n}$ on $\frac{A}{s_{n}}$ and let $\tau(A)$ be the mixing time of the trace of $K$ on $A$. Then
\be 
\lim_{n \rightarrow \infty} \tau_{n} &= \tau \\
\lim_{n \rightarrow \infty} \tau_{n}(A) &= \tau(A). \\
\ee  
\end{cor}

\begin{remark}
This corollary can be applied for chains $K_{n} = \KMHW{n}$  to obtain a bound of the form \eqref{IneqScalingAssumptionTradeoff1}. In the common situation that the mixing time $\tau$ in this corollary is infinite but $\tau(A)$ is finite for all compact sets $A$, this corollary gives a useful bound of the form \eqref{IneqScalingAssumptionTradeoff1} for a truncated copy of $\KMHW{n}$.
\end{remark}

It is often easy to obtain bounds of the form \eqref{IneqScalingAssumptionTradeoff2}, \eqref{IneqBiasAssumptionTradeoff} directly based on the form of the posterior:

\begin{lemma} \label{LemmaWideMhProps} 
Let $\KMH$ and $\KMHW{n}$ be the kernels associated with Algorithms \ref{AlgGenMH} and \ref{AlgMhWide},  and let $\PMH$ and $\PMHW{n}$ be their stationary distributions. 
\be \label{EqWideMhStat}
\PMHW{n}(\theta) \propto p(\theta) \prod_{i=1}^{N} p(\theta | x_{i})^{\frac{n}{N}}.
\ee 
\end{lemma}
\begin{proof}
This is clear from the definition of the Metropolis-Hastings chain $\KMHW{n}$.
\end{proof}

\section{Other Analyses and Future Work} \label{SecBigMisc}
We discuss the sharpness of our bounds and mention some area for future work.

\subsection{Large-Data Limit}

We introduce an algorithm that serves as a `large-data limit' of the kernel $\KSW{1}{n}$ defined by Algorithm \ref{AlgGenMhWide}, and discuss how this algorithm is related to our convergence bounds in Theorem \ref{ThmTradeoff} as well as a popular observation in the approximate MCMC community.

Fix the `true' parameter value $\theta^{\ast} \in \Theta$ and define the algorithm:

\begin{algorithm}[!ht] \caption{Infinite-Data Resampling MCMC}
\label{AlgGenMhInfinite}
\begin{algorithmic} 
\STATE Initialize $X_{1} = x$.
\FOR {$t = 1$ to $T$}
\STATE Generate $\ell_{t+1} \sim L(X_{t}, \cdot)$ and $u_{t} \sim \mathbb{U}[0,1]$. 
\STATE sample $x_{1}^{\ast},\ldots,x_{n}^{\ast}$  from $p(\cdot | \theta^{\ast})$, set $\psi(u_{t},X_{t},\ell_{t+1}) = \frac{1}{n} \log(u_{t} \frac{p(X_{t}) L(X_{t},\ell_{t+1})}{ p(\ell_{t+1}) L(\ell_{t+1},X_{t})})$, and set $\Lambda^{\ast}(X_{t},\ell_{t+1}) = \frac{1}{n} \sum_{i=1}^{n} \log(\frac{p(x_{i}^{\ast} | \ell_{t+1})}{p(x_{i}^{\ast} | X_{t})})$.
\IF {$\Lambda^{\ast}(X_{t},\ell_{t+1}) > \psi(u_{t},X_{t},\ell_{t+1})$}
\STATE  set $X_{t+1} = \ell_{t+1}$.
 \ELSE
\STATE
 set $X_{t+1} = X_{t}$.
\ENDIF
\ENDFOR
\end{algorithmic}
\end{algorithm}

Denote by $\KMHI{n}$ the kernel associated with Algorithm \ref{AlgGenMhInfinite}. Under modest assumptions, this kernel is the limit as $N$ goes to infinity of the Kernel $\KSW{1}{n}$ defined by Algorithm \ref{AlgGenMhWide}: 

\begin{lemma} [Limit of $\KSW{1}{n}$] \label{LemmaKernelLimits}
Let data points $x_{1},x_{2},\ldots$ be an i.i.d. sequence drawn from $p(x | \theta^{\ast})$. For $n \leq N \in \mathbb{N}$, let $\{ X_{t}\}_{t \geq 0}$ be a Markov chain driven from  the kernel $\KSW{1}{n}$ and associated with the ordered sequence of data points $\{ x_{i} \}_{i =1}^{N}$. Let $\{Y_{t}\}_{t \geq 0}$ be a Markov chain driven by the kernel  $\KMHI{n}$. Assume that $\{ X_{t} \}_{t \geq 0}$, $\{ Y_{t} \}_{t \geq 0}$ have the same starting point $X_{0} = Y_{0} = x \in \Omega$.  Then for fixed $T \in \mathbb{N}$ and any measurable $A \subset \Omega^{T+1}$,
\be 
\lim_{N \rightarrow \infty} \P[(X_{0},\ldots,X_{T}) \in A] = \P[(Y_{0},\ldots,Y_{T}) \in A].
\ee 
\end{lemma}

Lemma \ref{LemmaKernelLimits} implies that $\KSW{1}{n}$ is a small perturbation of $\KMHI{n}$ for all $N$ sufficiently large, while Lemma \ref{LemmaSmallPerturbation}  allowed us to conclude that $\KSW{A}{n}$ is a small perturbation of  $\KMHW{n}$ only for $A \gg n$. We point out that this means Lemma \ref{LemmaKernelLimits} can be used with Theorem \ref{ThmTradeoff} to obtain better convergence rates for approximate MCMC algorithms in the large-$N$ limit.

Recall that the purpose of  Lemma \ref{LemmaSmallPerturbation} was to obtain a bound of the form \eqref{IneqPerturbationAssumptionTradeoff}, to be used in an application of Theorem \ref{ThmTradeoff}. When it applies, Lemma \ref{LemmaKernelLimits} gives a bound of the form \eqref{IneqPerturbationAssumptionTradeoff} for kernel $\tilde{K}^{(n)} = \KSW{1}{n}$, rather than the more computationally expensive kernel $\KSW{A}{n}$ studied in Lemma \ref{LemmaSmallPerturbation}. This means that, in the large-$N$ limit under which Lemma \ref{LemmaKernelLimits} applies, it can be used to give a bound of the form \eqref{IneqPerturbationAssumptionTradeoff} for a computationally cheaper kernel, which in turn gives a better final convergence rate for Theorem \ref{ThmTradeoff}. 

\begin{remark} [Open Question]
We have proposed two families of `interpolating' chains, $\KSW{1}{n}$ and $\KMHI{n}$, to be used in Theorem \ref{ThmTradeoff}.  Can we construct a sequence of interpolating chains that gives better convergence rates, or are useful in a wider set of circumstances?
\end{remark}

In addition to these theoretical considerations, Lemma \ref{LemmaKernelLimits} raises some practical questions. The convergence rates given in Theorem \ref{ThmTradeoff} are, at best, the same as would be obtained by simply picking a fixed subsample of size $n \approx \sqrt{M}$ and running the usual Metropolis-Hastings algorithm for that subsample. We might ask: in the `big-data' limit that $N \rightarrow \infty$ and $M$ is large but fixed, why not simply pick a single subsample of the data and run the usual Metropolis-Hastings algorithm for that subsample? The fact that changing the subsample at every step is popular suggests that many practitioners believe it is a good idea, but to our knowledge there have been no attempts to justify it.

The following example suggests that the answer to this question is not completely straightforward:

\begin{example} [Does Resampling Help?]

Fix $n \in \mathbb{N}$. Let $p(\theta) = \mathcal{N}(0,\sigma^{2})$, let $p(x | \theta) = \mathcal{N}(x,1)$ and let $L(x,\cdot) \mathcal{N}(x, \frac{1}{n})$. Let $x_{1},\ldots,x_{n}$ be an i.i.d. sample from $\mathcal{N}(0,1)$, and let $K_{\mathrm{ss}}$ be the (random) kernel associated with Algorithm \ref{AlgGenMH} and these parameters and data. Let  $\KMHI{n}$ be the kernel associated with Algorithm \ref{AlgGenMhInfinite} and these parameters and data. Let $\pi_{\mathrm{ss}}$ and $\PMHI{n}$ be their stationary distributions. Define $f(x) = x^{2}$. It is easy to calculate $\pi_{\mathrm{ss}}$:
\be 
\pi_{\mathrm{ss}} = \mathcal{N}( (\sigma^{-2} + n)^{-1}( \sum_{i=1}^{n} x_{i}), (\sigma^{-2} + n)^{-1}).
\ee 
Integrating over $x_{1},\ldots,x_{n} \sim \mathcal{N}(0,1)$ gives
\be 
\E[\pi_{\mathrm{ss}}(f)] = n (\sigma^{-2} + n)^{-2} +  (\sigma^{-2} + n)^{-2} =  \frac{n+1} {(\sigma^{-2} + n)^{-2}}.
\ee 
It is harder to compute $\PMHI{n}$ directly. By simulation, we found that for $\sigma = 1$,
\be 
n \PMHI{n} (f) \geq 1.4
\ee 
for $n$ large, while $n \E[\pi_{\mathrm{ss}}(f)] \approx 1$. This suggests that, at least for this example, repeated resampling hurts performance. 
\end{example}

\begin{remark} [Open Question]
In what sense does repeat resampling help MCMC performance, relative to picking a single subsample at the start of the run?
\end{remark}

\subsection{Sharpness of Perturbation Bounds} \label{SecMisc}

We point out that the bounds in Section \ref{SecConvBounds} are essentially sharp, but can be misleading for statistical examples. To emphasize our interest in our perturbation bounds rather than the details of specific chains, both of our examples deal only with Markov chains that give i.i.d. samples from their target distributions. In both cases, we can calculate the approximation error exactly, rather than using approximations.

\begin{example} [Improved Convergence Rates]
Consider the target $\pi(\theta \vert \{x_{i} \}_{i = 1}^{N}) = \mathcal{N} \left( \frac{\sum_{i = 1}^{N} x_{i}}{N}, \frac{1}{N} \right)$ with approximation $\pi_{S}(\theta) = \mathcal{N} \left( \frac{\sum_{x \in S} x}{\vert S \vert}, \frac{1}{N} \right)$ for any subset (possibly with repetition) $S$ of $\{ x_{i} \}_{i=1}^{N}$. 

For any given computational budget $M$ and subsample size $n \leq M,N$, we drawn $\lfloor \frac{M}{n} \rfloor$ Monte Carlo samples to approximate $\theta$ as follows. Choose $n$ points $S = \{ x_{1}^{\ast},\ldots,x_{n}^{\ast} \}$ uniformly at random with replacement from $\{ x_{i} \}_{i=1}^{N}$. For $1 \leq t \leq \frac{M}{n}$, then draw $\theta_{t} \sim \pi_{S}(\theta)$. By the usual decomposition of variance formula, 
\be 
\var\Big( \frac{n}{M} \sum_{t=0}^{\frac{M}{n}} \theta_{t} \Big) &= \E \Big( \var \big( \frac{n}{M} \sum_{t=0}^{\frac{M}{n}} \theta_{t} \Big \vert S \big) \Big) + \var \Big( \big( \frac{n}{M} \sum_{t=0}^{\frac{M}{n}} \theta_{t} \Big \vert S_{t} \big) \Big) \\
&= \frac{n}{MN} + \frac{1}{n}.
\ee 
In this setting, choosing $n = \min(M, N, \sqrt{MN})$ is optimal, giving the usual $O \left( \frac{1}{\sqrt{M}} \right)$ convergence rate as the computational resources $M$ becomes large much more slowly than the amount of data $N$. Although the details change, similar conclusions hold if the set $S$ is resampled at each time $t$, and also if sampling is done without replacement. 
\end{example}

\begin{example} [Sharpness of Perturbation Bounds]
To find simple examples for which the $M^{-\frac{1}{4}}$ rate of convergence for estimates is correct, fix again a computational budget $M$. Also fix a constant $0 \leq n \leq M,N$ that will stand in for the subsample size. For now, we assume that $\frac{M}{n}$ is an integer, as this does not affect our conclusions. We define the measure $\mu_{n} = \left( 1 - \frac{1}{\sqrt{n}} \right) \mathrm{U}[0,1] + \frac{1}{\sqrt{n}} \delta_{0}$, and view $\mu_{n}$ as an approximation of $\mu_{\infty} = \mathrm{Unif}[0,1]$. We then consider a sequence of i.i.d. samples $\theta_{1}, \ldots, \theta_{ \frac{M}{n}}$ from $\mu_{n}$. We can exactly calculate the mean-square error obtained by using the Monte Carlo samples $\{ \theta_{t} \}_{t=1}^{\frac{M}{n}}$ from $\mu_{n}$ to estimate the mean of $\mu_{\infty}$:
\be 
\E \Big( \big( \frac{n}{M} \sum_{t=1}^{\frac{M}{n}} \theta_{t} - \frac{1}{2} \big)^{2} \Big) = \frac{1}{4n} + \frac{n}{3M}.
\ee 
The optimal choice is $n \approx \sqrt{\frac{3}{4}M}$, giving a decay rate of $O \left( M^{-\frac{1}{4}} \right)$. We can see here that the cancellation that occurs in the previous example does not occur here. 
\end{example}

This pair of examples shows that perturbation-theoretic bounds of the form given in Section \ref{SecConvBounds} cannot give a convergence rate that is much better than $M^{-\frac{1}{4}}$, but that certain (very simple) statistical examples seem to have much faster convergence rates.

\section*{Acknowledgements}
NSP is partially supported by NSF and ONR grants. AMS is partially supported by NSERC grants. We thank Christian Robert for comments on a previous version of this paper.

\section*{Appendix A: Proofs of Technical Results}

\subsection{Proofs for Section \ref{SecConvBounds}}

\begin{proof}[Proof of Lemma \ref{LemmaWassCont}] 
 By the triangle inequality and inequalities \eqref{IneqSimpleWassContraction}, \eqref{IneqKernCloseWassCont}, 
\be \label{IneqWassInitialTriangle}
W_{d}(K(x,\cdot), \tilde{K}(y,\cdot)) &\leq W_{d}(K(x,\cdot), K(y,\cdot)) + W_{d}(K(y,\cdot), \tilde{K}(y,\cdot)) \\
&\leq (1-\alpha) d(x,y) + \delta 
\ee 
for any $x,y \in \mathcal{X}$. Next, fix $\gamma > 0$ and assume that $\{ X_{t} \}_{t \geq 0}$ is started according to the stationary distribution (that is, $X_{0} \sim \pi$). By the definition of the Wasserstein distance, for any $0 \leq s < T$ it is possible to couple $X_{s+1}, \tilde{X}_{s+1}$ given $X_{s}, \tilde{X}_{s}$ so that 
\be 
\E[d(X_{s+1}, \tilde{X}_{s+1}) \vert X_{s}, \tilde{X}_{s}] \leq \gamma + W_{d}(K(X_{s},\cdot), \tilde{K}(\tilde{X}_{s},\cdot)).
\ee 
Combining this bound with inequality \eqref{IneqWassInitialTriangle} and \eqref{IneqThreePart}, it is possible to couple $X_{s+1}, \tilde{X}_{s+1}$ given $X_{s}, \tilde{X}_{s}$ so that 
\be \label{MainWassIneqCoupInX}
\E[d(X_{s+1}, \tilde{X}_{s+1}) \vert X_{s}, \tilde{X}_{s}] \leq \gamma + (1 - \alpha) d(X_{s},\tilde{X}_{s}) + \delta
\ee 
if $X_{s}, \tilde{X}_{s} \in \mathcal{X}$ and so that
\be \label{MainWassIneqCoupOutX}
\E[L(X_{s+1}) | X_{s}] &\leq (1 - \beta) L(X_{s}) + \mathcal{B} \\
\E[L(\tilde{X}_{s+1}) | \tilde{X}_{s}] &\leq (1 - \beta) L(\tilde{X}_{s}) + \mathcal{B}\\
\ee
otherwise. Define $\tau = \sup \{ t \geq 0 \, : \, X_{t}, \tilde{X}_{t} \in \mathcal{X} \}$. By the definition of the function $L$,
\be \label{IneqWassMainSimpleMainBound}
W_{d}(\mathcal{L}(\tilde{X}_{T}), \pi) &= \E[d(X_{T}, \tilde{X}_{T}) \textbf{1}_{\tau \geq T}] +  \E[d(X_{T}, \tilde{X}_{T}) (1 - \textbf{1}_{\tau \geq T})] \\
&\leq \E[d(X_{T}, \tilde{X}_{T}) \textbf{1}_{\tau \geq T}] +  \E[(L(X_{T}) + L(\tilde{X}_{T}) + 2 c_{p})(1 - \textbf{1}_{\tau \geq T})]
\ee 
To deal with the first term of inequality \eqref{IneqWassMainSimpleMainBound}, inequality \eqref{MainWassIneqCoupInX} implies
\be \label{IneqWassMainSimpleIterate} 
\E[d(X_{T}, \tilde{X}_{T}) \textbf{1}_{\tau \geq T}] &\leq \E[\delta + \gamma + (1-\alpha)\E[d(X_{T-1}, \tilde{X}_{T-1})\textbf{1}_{\tau \geq T}]] \\
&\leq \ldots  \\
&\leq \frac{\delta + \gamma}{\alpha} + (1-\alpha)^{T} E(x).
\ee 
To deal with the second term of inequality \eqref{IneqWassMainSimpleMainBound}, we have by inequality \eqref{IneqThreePart} (and eliding repeated use of the computation in inequality \eqref{IneqWassMainSimpleIterate}): 
\be \label{IneqWassMainSimpleIterate2} 
\E[(L(X_{T}) &+ L(\tilde{X}_{T}) + 2 c_{p}) (1 - \textbf{1}_{\tau \geq T})] \\
&\leq \P[1 \leq \tau \leq T-1] \sup_{1 \leq t \leq T-1} \E[(L(X_{T}) + L(\tilde{X}_{T}) + 2 c_{p}) \vert \tau = t]  + \E[(L(X_{T}) + L(\tilde{X}_{T}) + 2 c_{p}) \textbf{1}_{X_{0} \notin \mathcal{X}}] \\
&\leq 2 \P[1 \leq \tau < T-1] \sup_{1 \leq t \leq T} \left(  \mathcal{C} (1 - \beta)^{T-t} + \frac{\mathcal{B}}{\beta}  + c_{p} \right) \\
&+  \left( 2 \frac{\mathcal{B}}{\beta} + (1 - \beta)^{T} (L(x) + \mathcal{D}) + 2 c_{p} \right) \P[X_{0} \notin \mathcal{X}] \\
&=2 \P[1 \leq \tau < T-1] \left(  \mathcal{C}  + \frac{\mathcal{B} }{\beta} + c_{p} \right) \\
&+  \left( 2 \frac{\mathcal{B} }{\beta} + (1 - \beta)^{T} (L(x) + \mathcal{D}) + 2 c_{p} \right) \P[X_{0} \notin \mathcal{X}] \\
\ee 
Applying bounds \eqref{IneqWassMainSimpleIterate}, \eqref{IneqWassMainSimpleIterate2} to inequality \eqref{IneqWassMainSimpleMainBound} gives 
\be 
W_{d}(\mathcal{L}(\tilde{X}_{T}), \pi) &\leq \frac{\delta + \gamma}{\alpha} + (1-\alpha)^{T} E(x) + 2 \left( 1 - \P[\{ X_{t} \}_{t=1}^{T-1} \cup \{ \tilde{X}_{t} \}_{t =1}^{T-1} \subset \mathcal{X}] \right) \left(  \mathcal{C}  + \frac{\mathcal{B} }{\beta} + c_{p} \right) \\
&+  \left( 2 \frac{\mathcal{B} }{\beta} + (1 - \beta)^{T} (L(x) + \mathcal{D}) + 2 c_{p} \right) \pi(\mathcal{X}^{c}).
\ee 
Letting $\gamma$ go to 0 completes the proof of inequality \eqref{IneqAbsCoupResTwo}. 
\end{proof}

\begin{proof} [Proof of Corollary \ref{CorWassCont}]
Inequality \eqref{IneqAbsCoupResTwoNice} follows immediately from inequality \eqref{IneqAbsCoupResTwo}. Inequality follows from \eqref{IneqAbsCoupResTwoNice} by noting that 
\be 
W_{d}(\tilde{\pi}, \pi) = \lim_{T \rightarrow \infty} W_{d}(\mathcal{L}(\tilde{X}_{T}), \pi ) \leq \lim_{T \rightarrow \infty} ( \frac{\delta}{\alpha} + (1-\alpha)^{T} E(x)) = \frac{\delta}{\alpha}.
\ee 
\end{proof}

\begin{proof} [Proof of Corollary \ref{LemCoupIneqAdapt2}]
Applying the triangle inequality to inequality \eqref{IneqTvContAssumption} $\tmix$ times, we have 
\be 
\sup_{x \in \Omega} \| \tilde{K}^{\tmix}(x,\cdot) - K^{\tmix}(x,\cdot) \|_{\TV} < \tmix \, \delta.
\ee 
Thus, applying  Corollary \ref{CorWassCont} to the thinned chains $\{ X_{k \, \tmix} \}_{k \geq 0}$, $\{ \tilde{X}_{k \, \tmix} \}_{k \geq 0}$ yields 
\be \label{IneqAbsCoupResThreeInterm} 
\| \mathcal{L}( \tilde{X}_{k \, \tmix}) -  \pi \|_{\TV} \leq \frac{4 \delta}{3} \tmix + 4^{-k},
\ee 
which immediately implies inequality \eqref{IneqAbsCoupResThree}. To prove inequality \eqref{IneqAbsCoupResThreeStat} from  inequality \eqref{IneqAbsCoupResThree}, note 

\be 
\| \tilde{\pi} - \pi \| = \lim_{T \rightarrow \infty} \| \mathcal{L}(\tilde{X}_{T}) - \pi \| \leq \lim_{T \rightarrow \infty} ( \frac{4 \delta}{3} \tmix  + 4^{-\lfloor \frac{T}{\tmix} \rfloor}) = \frac{4 \delta}{3} \tmix.
\ee 

Finally, we prove inequality \eqref{IneqAbsCoupResThreeMix}. By inequalities \eqref{IneqAbsCoupResThreeInterm} and \eqref{IneqAbsCoupResThreeStat}, we have 
\be 
\| \mathcal{L}( \tilde{X}_{2 \tmix}) - \tilde{\pi} \|_{\TV} &\leq \| \mathcal{L}( \tilde{X}_{2 \tmix}) - \pi \|_{\TV} + \| \pi - \tilde{\pi} \|_{\TV} \\
&\leq \frac{4 \delta}{3} \tmix + 4^{-k} +  \frac{4 \delta}{3} \tmix \\ 
& < \frac{1}{4}.
\ee 
\end{proof}

We need the following bound for the proof of Theorem \ref{BiasGeomErg}:

\begin{lemma}[Drift Implies Concentration] \label{LemmaDriftConc}
Let $K$ be the transition kernel of a Markov chain satisfying inequality \eqref{EqLyapunovDef} and define $\mathcal{M}(C) = \{ x \in \Omega \, : \, V(x) \leq C \}$. Then
\be 
\pi(\mathcal{M}(C)) \geq  1 - \frac{b}{aC}.
\ee 

\end{lemma}
\begin{proof}
Let $\{ X_{t} \}_{t geq 0}$ be a Markov chain with transition kernel $K$, started at stationarity, \textit{i.e.}, $X_{0} \sim \pi$. Since $X_{\ell}$ then has distribution $\pi$ as well,
\be 
\pi(V) &= \E[\E[V(X_{\ell}) | X_{0}] ] \\
&\leq \E[(1-a)V(X_{0}) + b] \\
&= (1-a) \pi(V) + b.
\ee 
Thus, $\pi(V) \leq \frac{b}{a}$ and so by Markov's inequality,
\be 
\pi(\mathcal{M}(C)) \geq 1 - \frac{b}{aC}.
\ee 
\end{proof}

We then apply this:

\begin{proof} [Proof of Theorem \ref{BiasGeomErg}]
Let $\pi_{\mathcal{X}}$ and $\tilde{\pi}_{\mathcal{X}}$ be the stationary distributions of the traces of $\{X_{t}\}_{t \geq 0}$ and $\{ \tilde{X}_{t} \}_{t \geq 0}$ on $\mathcal{X}$; they are given by $\pi_{\mathcal{X}}(A) = \frac{\pi(A \cap \mathcal{X})}{\pi(\mathcal{X})}$ and  $\tilde{\pi}_{\mathcal{X}}(A) = \frac{\tilde{\pi}(A \cap \mathcal{X})}{\tilde{\pi}(\mathcal{X})}$. We have 
\be \label{BiasGeomErgMain}
\| \pi - \tilde{\pi} \|_{\TV} &= \sup_{A \subset \Omega} | \pi(A) - \tilde{\pi}(A) | \\
&= \sup_{A \subset \Omega} | \pi(A \cap \mathcal{X}) + \pi(A \cap \mathcal{X}^{c}) - \tilde{\pi}(A\cap \mathcal{X}) - \tilde{\pi}(A\cap \mathcal{X}^{c})| \\ 
&\leq \sup_{A \subset \Omega}( | \pi(A \cap \mathcal{X}) - \tilde{\pi}(A\cap \mathcal{X})| + \pi(A \cap \mathcal{X}^{c})  + \tilde{\pi}(A\cap \mathcal{X}^{c})) \\
&\leq  \sup_{A \subset \Omega} | \pi_{\mathcal{X}}(A \cap \mathcal{X}) \pi(\mathcal{X}) - \tilde{\pi}_{\mathcal{X}}(A\cap \mathcal{X}) \tilde{\pi}(\mathcal{X}) | + \pi(\mathcal{X}^{c})  + \tilde{\pi}(\mathcal{X}^{c}) \\
&\leq \sup_{A \subset \Omega} | \pi_{\mathcal{X}}(A \cap \mathcal{X}) - \tilde{\pi}_{\mathcal{X}}(A\cap \mathcal{X})  | + 2 \pi(\mathcal{X}^{c})  + 2 \tilde{\pi}(\mathcal{X}^{c}) \\
&= \| \pi_{\mathcal{X}} - \tilde{\pi}_{\mathcal{X}} \|_{\TV} + 2 \pi(\mathcal{X}^{c})  + 2 \tilde{\pi}(\mathcal{X}^{c}). 
\ee 

We now bound these three terms. By Corollary \ref{LemCoupIneqAdapt2}, 
\be \label{BiasGeomErgSmall}
\| \pi_{\mathcal{X}} - \tilde{\pi}_{\mathcal{X}} \|_{\TV} \leq \frac{4 \delta}{3} \tmix.
\ee 
By Lemma \ref{LemmaDriftConc},
\be \label{BiasGeomErgBig}
 \pi(\mathcal{X}^{c}) \leq \frac{b}{a C}, \qquad  \tilde{\pi}(\mathcal{X}^{c}) \leq \frac{\tilde{b}}{\tilde{a} C}.
\ee 
Applying bounds \eqref{BiasGeomErgSmall} and \eqref{BiasGeomErgBig} to inequality \eqref{BiasGeomErgMain} completes the proof. 
\end{proof}

\begin{proof} [Proof of Lemma \ref{LemmDriftCondAustFrame}]
Let $\{ X_{t} \}_{t \in \mathbb{N}}$ be a chain evolving according to $\tilde{K}$. Then
\be 
\E[V(X_{t+1}) \vert X_{t} = x] &= \int_{z} \left( \tilde{\alpha}(x,z) V(z) + (1 - \tilde{\alpha}(x,z)) V(x) \right) L(x,dz) \\
&\leq  \int_{z} \left( \alpha(x,z) V(z) + (1 - \alpha(x,z)) V(x) \right) L(x,dz) \\
&+ \delta \int_{z} \left(f(x,z) V(z) + f(x,z) V(x) \right) L(x,dz) \\
&\leq (1 - a)V(x) + b + \delta V(x) + \delta V(x) \\
&= (1 - a + 2 \delta)V(x) + b.
\ee 
\end{proof}

\subsection{Proofs from Section \ref{SecAusterity}}

\begin{proof} [Proof of Lemma \ref{LemmaBurninBound}]
Let $\tau_{\mathrm{coup}} = \inf \{ t \geq 0 \, : \, X_{t} = Y_{t} \}$ and let $T = \lfloor \frac{M}{c(K)}  \rfloor$. By the construction of the chains $\{X_{t} \}_{t \geq 0}$, $\{Y_{t} \}_{t \geq 0}$, we have 
\be 
\mathcal{E}_{M}^{(1)}(K,f)  &\leq \sup_{g \, : \, \Omega \mapsto [0,1]} \sup_{x \in \Omega}  \E_{x}[ (\frac{1}{T+1} \sum_{t=0}^{T} (g(X_{t}) - g(Y_{t})))^{2}] \\
&\leq  \sup_{x \in \Omega}  \E_{x}[ (\frac{1}{T+1} \sum_{t=0}^{T} \textbf{1}_{X_{t} \neq Y_{t}})^{2} ]  \\
&= \sup_{x \in \Omega} \E_{x} [(\frac{\tau_{\mathrm{coup}}}{T+1})^{2}] \\
&\leq \frac{1}{(T+1)^{2}} \sum_{k=0}^{\lceil \frac{T+1}{\tmix} \rceil} \tmix^{2} \sup_{x \in \Omega} \tmix \P[\tau_{\mathrm{coup}} > k \tmix] \\
&\leq \frac{1}{(T+1)^{2}} \sum_{k=0}^{\infty} \tmix^{2} (\sup_{x \in \Omega} \P[\tau_{\mathrm{coup}} > \tmix])^{k} \leq  \left( \frac{\tmix}{T+1} \right)^{2} \sum_{k=0}^{\infty} 2^{-k} \leq  2  \left( \frac{\tmix}{T+1} \right)^{2}.
\ee 
\end{proof}

\begin{proof} [Proof of Theorem \ref{ThmTradeoff}]
Set $n_{0} = \min \{ n \, : \, C''' \, n^{-c_{2}} \leq \frac{9}{128 C' \tmix} \}$ and fix $n_{0} \leq n \leq N$. We calculate: 

\be 
\mathcal{E}_{M}(\tilde{K}^{(n)}, f) &= \mathcal{E}_{M}^{(1)}(\tilde{K}^{(n)}, f) + \mathcal{E}_{M}^{(2)}(\tilde{K}^{(n)}, f) + \mathcal{E}_{M}^{(3)}(\tilde{K}^{(n)}, f) \\
&\leq 2 (\ttmix^{(n)})^{2} \big\lfloor \frac{M}{n^{1+c_{1}}} \big\rfloor^{-2} + 2 \ttmix^{(n)} \mathrm{Var}_{\pi^{(n)}}(f)  \big\lfloor \frac{M}{n^{1+c_{1}}} \big\rfloor^{-1} \\
&+ \frac{16}{9} (\ttmix^{(n)})^{2} \sup_{x \in \Omega} \| \tilde{K}^{(n)}( x,\cdot) - K^{(n)}(x,\cdot) \|_{\TV}^{2} + | \pi(f) - \pi^{(n)}(f) | \\
&\leq 2 (C' \tmix)^{2} \big\lfloor \frac{M}{n^{1+c_{1}}} \big\rfloor^{-2} + 2 C' \tmix \mathrm{Var}_{\tilde{\pi}^{(n)}}(f)  \big\lfloor \frac{M}{n^{1+c_{1}}} \big\rfloor^{-1} \\
&+ \frac{16}{9} (C' \tmix)^{2}  n^{-2 c_{2}} + C''' n^{-c_{2}} \\
&\leq 2 (C' \tmix)^{2} \big\lfloor \frac{M}{n^{1+c_{1}}} \big\rfloor^{-2} + 2 C' C'' \frac{\tmix}{n} \big\lfloor \frac{M}{n^{1+c_{1}}} \big\rfloor^{-1} + \frac{16}{9} (C' \tmix)^{2}  n^{-2 c_{2}} + C''' n^{-c_{2}},
\ee 
where the first inequality comes from Lemmas \ref{LemmaBurninBound}, \ref{LemmaVarianceBound} and \ref{LemmaBiasBound}, the second inequality comes from Corollary \ref{LemCoupIneqAdapt2} and assumptions \eqref{IneqPerturbationAssumptionTradeoff}, \eqref{IneqScalingAssumptionTradeoff1} and \eqref{IneqBiasAssumptionTradeoff},  and the last inequality comes from assumptions \eqref{IneqScalingAssumptionTradeoff1} and \eqref{IneqScalingAssumptionTradeoff2}.

Choosing $n = \lceil M^{\frac{1}{2+c_{1}}} \lceil$ to approximately optimize this bound gives inequality \eqref{IneqTradeoffMain}.

\end{proof}

We need the following in the proof of Lemma \ref{LemmaSmallPerturbation}:

\begin{lemma} \label{LemmaNonConcLogs}
Let $U \sim \mathrm{Unif}[0,1]$. Then for any $-\infty < a < b \leq 1$ with $|a -b | \leq 1$,
\be 
\P[\log(U) \in [a,b]] \leq \frac{3 e}{2} (b-a).
\ee 
\end{lemma}
\begin{proof}
We have 
\be 
\P[\log(U) \in [a,b]] &= \P[U \in [e^{a}, e^{b}]] \\
&= (e^{b} - e^{a}) = e^{b} (1 - e^{a-b}) \\
&\leq e^{b}( (b-a) + \frac{1}{2} (a-b)^{2}) \leq \frac{3 e}{2} (b-a).
\ee 
\end{proof}

We now apply it:

\begin{proof} [Proof of Lemma \ref{LemmaSmallPerturbation} ]
Fix $x \in \Omega$. We will construct a coupling of $\theta_{1} \sim   \KMHW{n}(x,\cdot) $ and $\theta_{2} \sim  \KSW{A}{n}(x,\cdot)$ and show that $\P[\theta_{1} \neq \theta_{2}] \leq \frac{2}{\sqrt{n}}$.

To construct $\theta_{1}$, let $\ell \sim L(x,\cdot)$ and $u \sim \mathrm{Unif}[0,1]$. Then define the functions 
\be \label{EqCouplingPertJust1}
\psi(x,y; v) &= \frac{1}{n} \log(v \frac{p(x) L(x,y)}{ p(y) L(y,x)}) \\
\Lambda(x,y; (z_{1},\ldots,z_{k})) &= \frac{1}{k} \sum_{i=1}^{k} \log(\frac{p(z_{i} | y)}{p(z_{i} | x)})
\ee 
Finally, we set
\be \label{EqMhRepSimp}
\theta_{1} &= \ell, \qquad \Lambda(x,\ell; (x_{1},\ldots,x_{N})) > \psi(x,\ell; u) \\
\theta_{1} &= x, \qquad \mathrm{otherwise.} \\
\ee 

Next, let $x_{1}^{\ast},\ldots,x_{An}^{\ast}$ be a size-$An$ sample without replacement from $x_{1},\ldots,x_{N}$. Using the same variables $\ell, u$ as in equation \eqref{EqMhRepSimp}, set
\be \label{EqCouplingPertJust2}
\theta_{2} &= \ell, \qquad \Lambda(x,\ell; (x_{1}^{\ast},\ldots,x_{An}^{\ast})) > \psi(x,\ell; u)  \\
\theta_{2} &= x, \qquad \mathrm{otherwise.} \\
\ee 

Let $\Lambda_{1} = \Lambda(x,\ell; (x_{1}^{\ast},\ldots,x_{An}^{\ast}))$, $\Lambda_{2} = \Lambda(x,\ell; (x_{1},\ldots,x_{N}))$ and $\Delta = \frac{1}{n} \log(\frac{p(x) L(x,y)}{ p(y) L(y,x)})$. Under the coupling given by equations \eqref{EqCouplingPertJust1} and \eqref{EqCouplingPertJust2}, 

\be 
\P[ & \theta_{1} \neq \theta_{2}] = \E[ \P[\theta_{1} \neq \theta_{2} | \ell]]\\ 
&=\E[ \int_{0}^{1}\P[ (\{ \Lambda(x,\ell; (x_{1}^{\ast},\ldots,x_{An}^{\ast})) > \psi(x,\ell; \lambda) \} \cap \{ \Lambda(x,\ell; (x_{1},\ldots,x_{N})) < \psi(x,\ell; \lambda)   \}) \\
& \cup ( \{\Lambda(x,\ell; (x_{1}^{\ast},\ldots,x_{An}^{\ast})) < \psi(x,\ell; \lambda) \} \cap  \{ \Lambda(x,\ell; (x_{1},\ldots,x_{N})) > \psi(x,\ell; \lambda)  \}) | \ell] d \lambda] \\
&\leq \E[ \int_{0}^{1} \P[ | \Lambda_{1} -  \Lambda_{2}| \geq |  \Lambda_{2} - \psi(x,\ell; \lambda) | \, |  \ell] d \lambda ] \\
&= \E[ \int_{0}^{1} \P[ | \Lambda_{1} - \Lambda_{2}| \geq |  \Lambda_{2} - \Delta - \frac{1}{n} \log(\lambda) | \, |  \ell] d \lambda ]. \\
\ee
 Using Theorem 1 of \cite{bardenet2015concentration}, this gives 
\be 
\P[ & \theta_{1} \neq \theta_{2}]  \leq \E[ \int_{0}^{1} \P[ | \Lambda_{1} -  \Lambda_{2}| \geq |  \Lambda_{2} - \Delta - \frac{1}{n} \log(\lambda) | \, |  \ell] d \lambda ] \\
&\leq \E[ \int_{0}^{1} \mathrm{exp}(-C^{-2} A n |  \Lambda_{2} - \Delta - \frac{1}{n} \log(\lambda) |^{2} ) d \lambda ] \\
&\leq \E[ \int_{\lambda \, : \, |  \Lambda_{2} - \Delta - \frac{1}{n} \log(\lambda) |^{2} > \frac{C^{2} \log(n)}{An} } \mathrm{exp}(-4 C^{-2} A n \frac{C^{2} \log(n)}{An} ) d \lambda ] 
+ \P[|  \Lambda_{2} - \Delta - \frac{1}{n} \log(u) |^{2} \leq \frac{C^{2} \log(n)}{An} ] \\
&\leq  \mathrm{exp}(-4 \log(n) )  +  \P[|  \Lambda_{2} - \Delta - \frac{1}{n} \log(u) |^{2} \leq \frac{C^{2} \log(n)}{An} ] \\
&\leq n^{-4} + \frac{1}{\sqrt{n}},
\ee 
where the last line follows from Lemma \ref{LemmaNonConcLogs}. Since 

\be 
\| \KSW{A}{n}(x,\cdot) - \KMHW{n}(x,\cdot) \|_{\TV}  \leq \P[\theta_{1} \neq \theta_{2}],
\ee 
this completes the proof.
\end{proof}

\begin{proof} [Proof of Theorem \ref{ThmUnifConv}]
Define 
\be 
\alpha(x,y) &= \min(1, \frac{\phi(y)}{\phi(x)}) \\
\alpha_{n}(x,y) &= \min(1, \frac{\phi_{n}(\frac{y}{s_{n}})}{\phi_{n}(\frac{x}{s_{n}})}),
\ee 
the (rescaled) acceptance probabilities of $K_{n}$ and $K$.

Fix a compact set $\mathcal{C} \subset \Theta$. We calculate

\be 
\lim_{n \rightarrow \infty} & \sup_{x,y \in \mathcal{C}} | \alpha(x,y) - \alpha_{n}(x,y) | = \lim_{n \rightarrow \infty} \sup_{x,y \in \mathcal{C}} | \min(1, \frac{\phi(y)}{\phi(x)}) - \min(1, \frac{\phi_{n}(\frac{y}{s_{n}})}{\phi_{n}(\frac{x}{s_{n}})}) | \\
&\leq \lim_{n \rightarrow \infty} \sup_{x,y \in \mathcal{C}} |  \frac{\phi(y)}{\phi(x)} -  \frac{\phi(y) + (\phi_{n}(\frac{y}{s_{n}}) - \phi(y))}{ \phi(x) + (\phi_{n}(\frac{x}{s_{n}}) - \phi(x))} | \\
&= \lim_{n \rightarrow \infty} \sup_{x,y \in \mathcal{C}} (|  \frac{\phi(y)}{\phi(x)} -  \frac{\phi(y) }{ \phi(x)} | + O(\frac{(\phi_{n}(\frac{y}{s_{n}}) - \phi(y))}{c(\mathcal{X})}) + O(\frac{(\phi_{n}(\frac{x}{s_{n}}) - \phi(x))}{c(\mathcal{X})})) \\
&= 0,
\ee 
where line 3 follows from part 5 of assumption \ref{DefConvDist} and line 4 follows from an application of Theorem 2 of \cite{cule2010theoretical} (whose assumptions are satisfied by parts 1-4 of assumption \ref{DefConvDist}). This limit implies 
\be \label{EqKernelConvCompact}
\lim_{n \rightarrow \infty} \sup_{x \in \mathcal{C}, A \subset \mathcal{C}} | K_{n}(\frac{ x}{s_{n}}, \frac{ A}{s_{n}}) - K(x,A) | = 0.
\ee 
For $M > 0$, let $\mathcal{B}(M) = \{ x \in \Theta \, : \, \| x \| \leq M \}$ and let $\tau(M) = \inf \{ t \geq 0 \, : \, Y_{t} \notin \mathcal{B}(M) \}$. Equation \eqref{EqKernelConvCompact} implies that, for any fixed $M > 0$, there exists a coupling of $\{X_{t}\}_{t \geq 0}$, $\{Y_{t} \}_{t \geq 0}$ that satisfies
\be \label{EqSeqConvCompact}
\lim_{n \rightarrow \infty} \P[(s_{n}X_{0}, \ldots,s_{n} X_{\min(T,\tau(M)-1)}) = (Y_{0}, \ldots,Y_{\min(T,\tau(M)-1)})] = 1.
\ee 
Since $T$ is finite, we have 
\be 
\lim_{M \rightarrow \infty} \P[\tau(M) \leq T] = 0.
\ee 
Combining this with equality \eqref{EqSeqConvCompact} gives 
\be 
\lim_{n \rightarrow \infty} \P[(s_{n}X_{0}, \ldots,s_{n} X_{T}) = (Y_{0}, \ldots,Y_{T})] = 1,
\ee 
completing the proof.
\end{proof}

\subsubsection{Proofs from Section \ref{SecMisc}}

\begin{remark}
The following proof shows convergence of sample paths whenever $n^{2} T^{2} = o(N)$, under essentially no assumptions. Under the stronger assumption that, for each fixed $\theta$, $p(\theta | x )$ is bounded away from 0 and infinity uniformly in $x$, it is possible to show convergence of sample paths under much weaker conditions (\textit{e.g.}  $n T = o(\frac{N}{\log(N)})$). 
\end{remark}

\begin{proof}  [Proof of Lemma \ref{LemmaKernelLimits}]
Fix $N, T$. Let $x_{1},x_{2},\ldots x_{N}$ be the sequence of data points used as parameters in $\KSW{1}{n}$. For $t \geq 0$ and $1 \leq i \leq n$, let $x_{t,i}^{\ast}$ be the subsample chosen in step 4 of Algorithm \ref{AlgGenMhWide} and define $x_{nt + i}' = x_{t,i}^{\ast}$. Let $\sigma_{nt + i} \in \{1,2,\ldots,N\}$ satisfy $x_{nt + i}' = x_{\sigma_{nt + i}}$. Without loss of generality, we can assume that the points $x_{1},\ldots,x_{N}$ are ordered so that $\sigma_{1} = 1$ and
\be \label{EqSigmaStuff}
\sigma_{j+1} \leq \max_{1 \leq i \leq j} \sigma_{i}
\ee 
for all $j \geq 1$.

Next, for $t \geq 0$ and $1 \leq i \leq n$, let $y_{t,i}$ be the $n$ random variables chosen in step 4 of Algorithm \ref{AlgGenMhInfinite} and define $y_{tn + i}' = y_{t,i}$. Finally, we couple the sequences $x_{1},x_{2},\ldots x_{N}$ and $y_{1}',y_{2}',\ldots$ so that $x_{i} = y_{i}'$ for $1 \leq i \leq N$. This is possible, since both sequences are i.i.d. sequences drawn from the same distribution.

Next, we note that if $\sigma_{n(T+1)} = n(T+1)$, then by inequality \eqref{EqSigmaStuff} we also have $\sigma_{j} = j$ for all $1 \leq j \leq n(T+1)$. However, if $\sigma_{j} = j$ for all $1 \leq j \leq n(T+1)$, we also have $x_{t,i}^{\ast} = y_{t,i}$ for all $0 \leq t \leq T$ and $1 \leq i \leq n$. Finally, conditional on  $x_{t,i}^{\ast} = y_{t,i}$ for all $0 \leq t \leq T$ and $1 \leq i \leq n$, it is clear that  $\mathcal{L}(X_{0},\ldots,X_{T}) =\mathcal{L}(Y_{0},\ldots,Y_{T})$. Thus, we have 
\be 
| \P[(X_{0},\ldots,X_{T}) \in A] - \P[(Y_{0},\ldots,Y_{T}) \in A] | &\leq 2 \P[\sigma_{n(T+1)} \neq n(T+1)] \\
&\leq \sum_{t=0}^{n(T+1)} \frac{t}{N} \leq \frac{n^{2} (T+1)^{2}}{N},
\ee 
and so
\be 
\lim_{N \rightarrow \infty} |\P[(X_{0},\ldots,X_{T}) \in A] - \P[(Y_{0},\ldots,Y_{T}) \in A]| \leq \lim_{N \rightarrow \infty} \frac{n^{2} (T+1)^{2}}{N} = 0,
\ee 
completing the proof.
\end{proof}

\bibliographystyle{alpha}
\bibliography{CBMBIB}
\end{document}